\documentclass[11pt, final]{amsart}
\usepackage[english]{babel}
\usepackage[utf8]{inputenc}
\usepackage{exscale}
\usepackage[centertags]{amsmath}
\usepackage{amsfonts,amstext,amssymb,amsthm, amsmath}
\usepackage[dvipsnames]{xcolor}
\usepackage[colorlinks=true,urlcolor=blue,citecolor=blue,linkcolor=blue,linktocpage,pdfpagelabels,bookmarksnumbered,bookmarksopen,backref=page]{hyperref}
\usepackage{comment}
\usepackage{newlfont}
\usepackage{enumerate}
\usepackage{color}
\usepackage{url}
\usepackage[marginparwidth=2cm]{geometry} 
\usepackage{float}
\usepackage{multirow}
\usepackage{graphicx}
\usepackage{amssymb}
\usepackage{verbatim} 
\usepackage{color}
\usepackage{epstopdf}
\usepackage{epsfig}
\usepackage[all]{xy}
\usepackage{mathrsfs}
\usepackage{stmaryrd}
\SetSymbolFont{stmry}{bold}{U}{stmry}{m}{n}
\usepackage{caption}
\usepackage{subcaption}
\usepackage{graphicx,wrapfig}
\usepackage{xfrac}
\usepackage{hyperref}
\usepackage{mathtools}
\usepackage{aligned-overset}
\usepackage[capitalize]{cleveref}
\usepackage{upgreek}
\numberwithin{equation}{section}
\usepackage{indentfirst}
\usepackage{tikz}
\usepackage{lmodern}
\usepackage{anyfontsize}
\usepackage{xfrac, nicefrac}
\usepackage{nccmath, amsmath}

\theoremstyle{definition}
\newtheorem{definition}{Definition}[section]
\newtheorem*{definition*}{Definition}
\newtheorem*{thm*}{Theorem}

\newtheorem{remark}[definition]{Remark}

\theoremstyle{plain}
\newtheorem{theorem}[definition]{Theorem}
\newtheorem{proposition}[definition]{Proposition}

\newtheorem{lemma}[definition]{Lemma}

\newcommand{\cB}{\mathcal{B}}

\newcommand{\per}[2]{\mathrm{P}\!\left(#1;#2\right)}
\newcommand{\bPhi}{\boldsymbol{\Phi}}
\newcommand{\rel}{\partial_{\mathrm{rel}}}
\newcommand{\cH}{{\mathcal{H}}}

\newcommand{\R}{\mathbb{R}}
\newcommand{\bbS}{\mathbb{S}}

\newcommand{\sing}{\mathrm{sing}}
\newcommand{\reg}{\mathrm{reg}}
\newcommand{\spt}{\mathrm{spt}}
\newcommand{\rmd}{\mathrm{d}}
\newcommand{\sdist}{\rmd_{\omega}}
\newcommand{\aniEn}{\bar{\nu}_H^\Phi}
\newcommand{\bU}{\mathbf{U}}

\newcommand{\bp}{{\mathbf{p}}}
\newcommand{\cW}{\mathcal{W}}

\newcommand\res{\mathop{\hbox{\vrule height 7pt width .3pt depth 0pt
\vrule height .3pt width 5pt depth 0pt}}\nolimits}

\newcommand{\udens}[3]{\Theta^*_{#1}(#2,#3)}
\newcommand{\ldens}[3]{\Theta_*^{#1}(#2,#3)}
\newcommand{\dens}[3]{\Theta^{#1}(#2,#3)}
\newcommand{\var}{\mathrm{var}}

\newcommand{\reach}{\mathrm{reach}_{\omega}}

\newcommand{\Om}{\Omega}
\newcommand{\om}{\omega}
\newcommand{\na}{{\nabla}}
\newcommand{\pa}{{\partial}}
\newcommand{\PEn}{{\mathcal{E}_{\omega}}}

\renewcommand{\div}{{\text{div}}}

\newcommand{\nustar}{e_n}

\def\e{\epsilon}
\def\w{\omega}
\def\TVF{\mathfrak{X}_T(\overline{H})}
\def\Fv{\delta_{\bPhi}}

\makeatletter
\def\@tocline#1#2#3#4#5#6#7{\relax
  \ifnum #1>\c@tocdepth 
  \else
    \par \addpenalty\@secpenalty\addvspace{#2}%
    \begingroup \hyphenpenalty\@M
    \@ifempty{#4}{%
      \@tempdima\csname r@tocindent\number#1\endcsname\relax
    }{%
      \@tempdima#4\relax
    }%
    \parindent\z@ \leftskip#3\relax \advance\leftskip\@tempdima\relax
    \rightskip\@pnumwidth plus4em \parfillskip-\@pnumwidth
    #5\leavevmode\hskip-\@tempdima
      \ifcase #1
       \or\or \hskip 1em \or \hskip 2em \else \hskip 3em \fi%
      #6\nobreak\relax
    \hfill\hbox to\@pnumwidth{\@tocpagenum{#7}}\par
    \nobreak
    \endgroup
  \fi}
\makeatother

\title{\bf{Rigidity of critical points of hydrophobic capillary functionals}}
\author{A. De Rosa, R. Neumayer, and R. Resende}
\begin{document}

\date{}

\begin{abstract}
We prove the rigidity, among sets of finite perimeter, 
of volume-preserving critical points of the capillary energy in the half space, in the case where the prescribed interior contact angle is between $90^\circ$  and $120^\circ$. No structural or regularity assumption is required on the finite perimeter sets.
Assuming that the ``tangential'' part of the capillary boundary is $\mathcal{H}^n$-null, this rigidity theorem extends to the full hydrophobic regime of interior contact angles between $90^\circ$  and $180^\circ$.
 Furthermore, we establish the anisotropic counterpart of this theorem under the assumption of lower density bounds.
\end{abstract}

\maketitle

\section{Introduction}

An equilibrated liquid droplet on a solid substrate exhibits capillarity: the cohesive forces driven by surface tension and adhesive forces between the liquid and solid substrate together cause the droplet to adhere to the substrate at a certain angle determined by Young's law. Analogous phenomena emerge for the equilibrated shapes of small solid particles, e.g. silver films \cite{WINTERBOTTOM1967303}, on a substrate.
Starting from the work of Gauss, the shapes of these equilibria have been modeled by critical points $E\subset \R^{n+1}$ ($n=2$ being the physical case) of free energy functionals subject to a volume constraint. In the absence of gravity and for a flat, homogeneous substrate $H:=\{x_{n+1}>0\}\subset \R^{n+1}$, the free energy functional {for a liquid droplet }is
\begin{align}\label{e: iso cap}
	 \PEn(E) := P(E; H) - \omega\,  P(E;\partial H).
\end{align}
Here, the adhesion coefficient $\omega \in (-1,1)$ is determined by the molecular makeup of the liquid and the substrate, and the relative perimeter $P(E; A)$ of a set of locally finite perimeter $E$, defined in Section~\ref{sec: prelimininaries}, is simply equal to $\mathcal{H}^n(\partial E \cap A)$ when the boundary $\partial E$ is sufficiently smooth.

 The goal of the paper is to classify {finite-volume} sets of finite perimeter $E$ in $\R^{n+1}$ that are critical points of $\PEn$---or of its anisotropic counterpart defined in \eqref{E:ani capillary energy SOFP}---with respect to volume-preserving variations. By this we mean $\frac{d}{dt}|_{t=0}\ \PEn(\Psi_t(E)) =0$ for every smooth one-parameter family $\{\Psi_t\}_{t\in (-\e,\e)}$ of diffeomorphisms on $\overline{H}$ such that $\Psi_t(\partial H) = \partial H$ with $|\Psi_t(E)| =|E|.$ This stationarity condition reads,  for some $\lambda>0$, 
\begin{equation}
	\label{eqn: isotropic stationary}
0 = \int_{\pa^* E \cap H} \div_E X - \omega  \int_{\pa^*E \cap \pa H} \div_H X  - \lambda \int_{\pa^*E \cap H} X \cdot \nu_E, \qquad \forall X \in \TVF.
\end{equation}
Here and in the sequel, $\TVF$ is set of vector fields $X \in C^1_c(\overline{H}, \R^{n+1})$ with $X\cdot\nu_H = 0$ on $\pa H$, $\nu_H = -e_{n+1}$ is the outer unit normal to $H$, and $\partial^* E$ is the reduced boundary of $E$ (see Section~\ref{sec: prelimininaries}).
If $E$ satisfies \eqref{eqn: isotropic stationary} and the relative boundary $\rel E := \overline{\partial E\cap H}$ is a smooth hypersurface with smooth boundary $\Gamma :=  \rel E \cap \partial H$, then $\partial E \cap H$ has constant mean curvature
and Young's law holds: $\rel E$ meets $\partial H$ with constant interior contact angle $\theta= \arccos \omega$.

Our first main result is the following unconditional classification of critical points of \eqref{e: iso cap} 
{in the regime $\omega \in (-1/2, 0]$}.

\begin{theorem}\label{T:Alexandrov:Isotropic}
Let $\omega \in (-{1/2},0]$ and $E\subset H$ be a set of finite perimeter with finite volume satisfying \eqref{eqn: isotropic stationary}. Then $E$ is a finite union of open balls and intersections with $H$ of open balls meeting $\partial H$ with interior contact angle $\arccos(\omega)$, all pairwise disjoint and with the same radii.
\end{theorem}
We always choose a representative for a set of finite perimeter with $\partial E = \overline{\partial^*E}$. If $|{\partial E}|=0$, which in particular holds for $E$ as in Theorem~\ref{T:Alexandrov:Isotropic} (as discussed in Remark~\ref{r:openrep}), then we choose a canonical representative of $E$ by an open set, see Section~\ref{sec: prelimininaries}. 
It is well known, since the work \cite{Delaunay1841} of Delaunay, that the finite volume assumption for $E$ is necessary for the validity of Theorem~\ref{T:Alexandrov:Isotropic}. Indeed, there are several examples of non-compact surfaces with constant mean curvature. They can be easily used to produce non-compact  capillary surfaces satisfying \eqref{eqn: isotropic stationary}.

The restriction $\omega >-1/2$ {in Theorem~\ref{T:Alexandrov:Isotropic}} is used only to prove that the set of points 
\begin{equation}\label{D:tangential-touching}
    T_E:= \left\{ p\in \Gamma\cap\partial^* E: \nu_E(p) = \nu_H\right\}
\end{equation} 
in $\Gamma$ with tangent plane equal to $\partial H$ has $\mathcal{H}^n$ measure zero. We in fact show that $\mathcal{H}^n(\Gamma)=0$ in this case, see Lemma~\ref{L:zerotouch}. {The set $T_E$ arises} naturally in the context of rigidity of critical points of \eqref{e: iso cap}, since for instance a ball in $H$ meeting $\partial H$ tangentially is a stationary configuration, though not a minimizer, for this problem. If one simply assumes $\mathcal{H}^n(T_E)=0$, then Theorem~\ref{T:Alexandrov:Isotropic} generalizes to the full hydrophobic regime $\omega \in (-1,0]$: 
\begin{theorem}\label{T:new theorem}
    Let $\omega \in (-1,0]$ and let $E \subset H$ be a set of finite perimeter with finite volume satisfying \eqref{eqn: isotropic stationary} and $\mathcal{H}^n(T_E) = 0$. 
Then the conclusion of Theorem~\ref{T:Alexandrov:Isotropic} holds.
\end{theorem}

This rigidity theorem was previously known for restricted classes of sets $E$ satisfying certain  regularity assumptions. Classical work of Wente \cite{wente1980symmetry} characterized critical points of the capillary energy, for the full range $\omega \in (-1,1)$, assuming  $\rel E$ and $\Gamma$ are smooth. More recently,
Xia-Zhang \cite{xia2023alexandrov} were able to extend Wente's result to sets $E$ such that the capillary boundary $\Gamma$ is smooth and the singular set of $\rel E$  is $\cH^{n-1}$-null.
These smoothness assumptions are needed in order to be able to integrate by parts and to deduce the pointwise validity of Young law from the criticality condition {\eqref{eqn: isotropic stationary}}.  Moreover, smoothness is also used to analyze the structure of blowups at every capillary boundary point as described in \cite[Remark 1.6]{xia2023alexandrov}.

Theorem~\ref{T:Alexandrov:Isotropic} is the first rigidity theorem {in the literature} for the capillary problem that applies unconditionally to the whole class of sets of finite perimeter. To achieve this, we first prove {the rigidity} for sets of  finite perimeter that {are subsolutions to} Young's law in a viscosity sense. 
Then, we show that any set satisfying the distributional stationarity condition {\eqref{eqn: isotropic stationary} } already {is a subsolution to} Young's law in a viscosity sense by using geometric measure theoretic arguments to characterize the blow-ups at capillary boundary points where $E$ is touched from the interior {by a smooth set $G$}. 
A key point here is that we can leverage the touching set {$G$} to prove regularity at such points.

Our third main theorem is Theorem~\ref{T:Alexandrov:Anisotropic} {below}, which is an anisotropic counterpart of Theorem~\ref{T:new theorem}. In the anisotropic setting, rigidity among sets $E$ such that  $\rel E$ and $\Gamma$ are both smooth was shown in \cite{jia2023alexandrov}.
As explained in the previous paragraph, blowup analysis is one of the crucial tools to remove {such smoothness} assumptions. Thus, our extension of Theorem~\ref{T:new theorem} to the anisotropic setting may be surprising, as critical points of anisotropic functionals famously lack a monotonicity formula \cite{characterizationareaallard}, meaning that generally speaking blowups may not exist, and when they do, they need not be cones. 

More concretely, let $\Phi \in C^{2,\alpha}(\R^{n+1}\setminus \{0\}, \R_+)$ be an elliptic integrand; see Section~\ref{sec: prelimininaries} for the precise definition. Its associated surface energy $\int_{\partial^*E} \Phi(\nu_E) \, \rmd \mathcal{H}^{n-1}$, used in the modeling of solid crystals with sufficiently small grains \cite{Herring}, is minimized among sets in $\R^{n+1}$ of a fixed volume by translations and dilations of the Wulff shape $\cW_1 \subset \R^{n+1}$.  
When a particle lies on a flat, homogeneous substrate in the absence of gravity, equilibria are modeled by critical points of the free energy functional 
\begin{equation}\label{E:ani capillary energy SOFP}
	\bPhi_{\omega}(E) := \int_{\partial^* E\cap H}\Phi(\nu_E)\,\rmd \cH^n - \omega P(E;\partial H).
\end{equation}
Here $\omega\in(-\Phi(\nu_H) , \Phi(\nu_H))$ is again determined by the makeup of the solid and the substrate.
A set of finite perimeter $E\subset H $ with finite volume is  a critical point of \eqref{E:ani capillary energy SOFP} under volume preserving variations if and only if it satisfies the distributional condition 
\begin{equation}\label{eqn: aniso isotropic stationary}
    0 = \int_{\pa^* E \cap H} B_{\bPhi}(\nu_E ) :DX\,\rmd \cH^n  - \omega  \int_{\pa^*E \cap \pa H} \div_H X\,\rmd \cH^n   - \lambda \int_{\pa^*E \cap H} X \cdot \nu_E\,\rmd \cH^n,
\end{equation}
for some $\lambda>0$ and for every $X \in \TVF$. Here $B_\Phi(\nu) := \Phi(\nu) \text{Id}  -\nu \otimes \nabla\Phi(\nu)$ for $\nu \in \bbS^n$.
In {the} case $E$ is a smooth set satisfying \eqref{E:ani capillary energy SOFP}, we can integrate by parts to show that $E$ has constant anisotropic mean curvature  $H_E^\Phi(x) = \lambda$ for $x \in \partial E\cap H$ and the Cahn-Hoffman normal $\nu_E^\Phi(x) := \nabla \Phi (\nu_E(x))$ to $E$ satisfies the anisotropic form of Young's law:
\begin{equation}\label{eqn: anisotropic Young's law}
    \nu_E^\Phi(x) \cdot\nu_H = -{\omega}, \quad \text{ for } x \in \Gamma .
\end{equation}
Among sets of finite perimeter in $\R^{n+1}$ of a fixed volume, the energy $\bPhi_\omega$ is minimized by a Winterbottom shape of radius $r>0$:
\[
 \cW_r^{\omega} := H\cap \left( \cW_r + {r}\omega\nu_H\right) + p,
\]
where $\cW_r = r\cW_1$ and $p \in\partial H$.

\begin{theorem}\label{T:Alexandrov:Anisotropic}
Let $\Phi \in C^{2,\alpha}(\R^{n+1}\setminus \{0\}, \R_+)$ be 
an elliptic integrand and $\omega \in (-\Phi(\nu_H), 0]$. Assume that $E\subset H$ is a  finite perimeter set with finite volume that satisfies \eqref{eqn: aniso isotropic stationary},
\begin{align}	
	\cH^n\left(T_E\right) &= 0,\label{A:RegAnis:tang touching pts}\\
    \cH^n\left(\left(\partial E\setminus \partial^* E\right)\cap H\right) &= 0,\label{A:RegAnis:Intro}
\end{align}
and for some $\epsilon>0$,
\begin{align}
   \frac{\mathcal{H}^{n}(\partial^* E \cap B_r(x))}{\om_n r^n} &\geq \epsilon \text{ for  $\mathcal{H}^n$- a.e. } x \in B_\epsilon(\Gamma) \cap \partial^*E\cap H \text{ and }r<x_{n+1}. \label{eqn: lower density assumption intro 1}
\end{align}
Then $E$ is a finite union of $\omega$-Winterbottom shapes and Wulff shapes, all pairwise disjoint and with the same radii.
\end{theorem}
In \eqref{eqn: lower density assumption intro 1},  $B_\e(\Gamma)$ denotes a tubular neighborhood of radius $\e$ of $\Gamma$. We explicitly observe that in Theorem~\ref{T:Alexandrov:Anisotropic} we cover the full hydrophobic regime,  since for $\omega \leq -\Phi(\nu_H)$ the variational problem for \eqref{E:ani capillary energy SOFP} with a volume constraint is simply equivalent to the boundaryless case, i.e. the standard (anisotropic) isoperimetric problem, and $\omega=0$ corresponds to the free boundary case.  As stated above, among sets with $\rel E$ and $\Gamma$ smooth, this conclusion was shown in \cite{jia2023alexandrov}.

\begin{remark}\label{rmk:arearedundant2}
    For the area functional (i.e., $\Phi(x) \equiv |x|$),  the conditions \eqref{A:RegAnis:Intro} and \eqref{eqn: lower density assumption intro 1}, as well as \eqref{A:RegAnis:tang touching pts} for $\omega \in (-1/2,0]$, are consequences of the monotonicity formula.
\end{remark}

\begin{remark}
The analogue of assumption \eqref{A:RegAnis:Intro} is  the condition imposed by De Rosa-Kolasi{\'n}ski-Santilli in \cite{de2020uniqueness} to characterize the critical points of the anisotropic isoperimetric problem. This assumption, as well as \eqref{eqn: lower density assumption intro 1},  is needed due to the lack of lower density bounds for anisotropic stationary varifolds. Proving lower density bounds is one of the major open problems in the regularity theory for anisotropic minimal surfaces.
\end{remark}

\begin{remark}
    Theorem \ref{T:Alexandrov:Anisotropic} can be applied to characterize local minimizers of the anisotropic capillary problem on closed manifolds with boundary, in the small volume regime. The proof would follow the same line of arguments used for the local minimizers of the anisotropic isoperimetric problem in closed manifolds in \cite{DeRosaNeumayer2023local}, replacing the use of the rigidity theorem \cite[Corollary 6.8]{de2020uniqueness} with Theorem \ref{T:Alexandrov:Anisotropic}.
\end{remark}

\begin{remark}
    The characterization proved in Theorem \ref{T:Alexandrov:Anisotropic} could be applied in characterizing the critical points of the anisotropic $2$-bubble problem, under convexity assumption of the chambers, in the spirit of \cite{DeRosaTione2025}. Indeed in this case each chamber would be a critical point for an appropriate $\bPhi_{\omega}$, with respect to the hyperplane separating the two chambers.
\end{remark}

\subsection{Discussion of the proofs}

The proofs of Theorem~\ref{T:Alexandrov:Isotropic}, Theorem~\ref{T:new theorem}, and Theorem~\ref{T:Alexandrov:Anisotropic}  are based on proving a Heintze-Karcher type inequality following \cite{jia2022heintze, xia2023alexandrov, jia2023alexandrov}, which in turn are in the spirit of a general scheme introduced by Montiel and Ros \cite{montiel1991compact}.  This basic approach  has proved to be powerful for proving geometric rigidity results in the smooth case \cite{montiel1991compact, BrendleCMC, HeLiMaGe, WangSpacetime, jia2022heintze, jia2023alexandrov, Hyperbolic}, {and} consists of exploring the set $E$ by following the normal bundle.
Moreover, it has been understood that for problems without a capillary boundary, 
this proof scheme can be made to work for a critical point $E$ with potential singularities,  provided  $E$ has almost-everywhere regular boundary and has bounded mean curvature in a viscosity sense
\cite{delgadino2019alexandrov,de2020uniqueness, maggi2023rigidity}.

In the capillary setting, there have been several formulations of Heintze-Karcher type inequalities, see for instance \cite{delgadino2024heintze, jia2023heintze, jia2022heintze, xia2023alexandrov, jia2023alexandrov}.  Apart from \cite{xia2023alexandrov}, these results all deal with smooth sets.  The formulations in \cite{jia2022heintze, xia2023alexandrov, jia2023alexandrov} are more suitable for relaxing regularity assumptions. However, as already mentioned, previous approaches to relaxing regularity assumptions \cite{xia2023alexandrov} still require {enough smoothness} to integrate by parts to leverage Young's law in a classical sense. Nevertheless, we introduce a suitable notion of viscosity subsolution to Young's law,  and inspired by technical refinements of \cite{delgadino2019alexandrov, de2020uniqueness, maggi2023rigidity, xia2023alexandrov}, we prove a Heintze-Karcher inequality for these viscosity subsolutions in Theorem~\ref{thm: HK} and a corresponding rigidity statement in Section~\ref{S:proof of Alexandrov}.

In view of Theorem~\ref{thm: HK}, one of the main challenges and new contributions of the paper is to show that any set of finite perimeter $E$ that is a critical point of $\PEn$ or $\bPhi_{\omega}$ with respect to volume-preserving variations is also a viscosity subsolution to Young's law -- unconditionally in the case of $\PEn$ and under mild density assumptions in the case of $\bPhi_{\omega}$; see Theorem~\ref{thm: stationary implies viscosity} and Remark~\ref{rmk:arearedundant2} below.

We say an open set $G\subset \R^{n+1}$ with $C^1$ boundary touches $E \subset  H$ from the interior at a point $p \in \Gamma$ if $G\cap H\subset E$ and $p\in\pa G$.
 We say that $E$ is a {\it viscosity subsolution} to Young's law if, for any $p \in \Gamma$ and $C^1$ set $G$ that  touches $E$ from the interior at $p$,
        \[\mbox{either } \qquad p \in T_E \qquad \mbox{or }\qquad
        \nu^\Phi_G(p) \cdot \nu_H \leq -\omega\,.
        \]
{Recall that the set $T_E$ was defined in \eqref{D:tangential-touching} and observe that t}his definition makes sense in view of \eqref{eqn: anisotropic Young's law}.

\begin{theorem}\label{thm: stationary implies viscosity} 
Let $\Phi \in C^{2,\alpha}(\R^{n+1}\setminus \{0\}, \R_+)$ be an elliptic integrand, $\omega \in (-\Phi(\nu_H), 0]$, and $E \subset H$ be a set of finite perimeter satisfying, for some $\epsilon>0$, 
\begin{equation}\label{eqn: Lower Density}
    \frac{\mathcal{H}^{n}(\partial^* E \cap B_r(x))}{\om_n r^n} \geq \epsilon \text{ for all }x \in B_\epsilon(\Gamma) \cap \partial^*E \text{ and }r<x_{n+1}.
    \end{equation}
 If $E$ satisfies \eqref{eqn: aniso isotropic stationary}, then $E$ is a viscosity subsolution for Young's law.
\end{theorem}

 As discussed above, in the absence of enough regularity to derive Young's law in a classical sense for critical points of $\PEn$ (or, more generally, of $\bPhi_{\omega}$), we instead argue locally through blowup analysis.
 A key point is that we need only to analyze blowups at points at which $E$ is touched from the interior by a $C^1$ set $G$, and the very presence of $G$ forces more regularity.

 In the simpler case of the area functional, the basic idea of the proof of Theorem~\ref{thm: stationary implies viscosity} is the following contradiction argument. 
 If there is a $C^1$ set $G$ that touches $E$ from the interior at $p \in\Gamma$ with ``too big'' of a contact angle, i.e. $\nu_G(p) \cdot \nu_H > - \omega$, then any blowup of $\partial^*E$ at $p$, when intersected with $H$, is a cone contained in the wedge $H\setminus T_pG.$ A maximum principle argument by iteratively tilting the half-hyperplane  bounding the wedge until it meets this blowup shows that the blowup of $\partial^*E$ in $H$ is a finite union of half-hyperplanes. As the blow-up is also stationary, integrating by parts in \eqref{eqn: isotropic stationary} for the blowup yields a contradiction.  This is carried out in Section~\ref{sec: stat to visc}. This is the part of the argument that fundamentally relies on being in the hydrophobic regime $\omega \in (-1,0]$, as Young's law for a finite union of half-hyperplanes meeting along a common $(n-1)$-plane only has a unique solution in the hydrophobic case; see Remark~\ref{R:small-contact-angle-regime} for further discussion of this point. This observation shows that proving Theorem~\ref{thm: stationary implies viscosity} and more generally Theorems \ref{T:Alexandrov:Isotropic}-\ref{T:new theorem} in the hydrophilic regime cannot solely rely on the local analysis of the blow-ups at the capillary boundary, requiring instead a global approach.

The monotonicity formula plays a central role in the strategy described above. Consequently, several steps break down in the anisotropic setting.
Indeed, blow-up analysis for anisotropic minimal surfaces is well-known to be challenging due to the absence of a monotonicity formula.
First, the existence of the blow-ups themselves is not guaranteed, due to the possible existence of infinite density points. 
Next, when blow-ups exist, they are not necessarily cones. Moreover, the lack of upper semi-continuity of the density leads to compactness issues for integral varifolds, hence blowups are not necessarily integral.

Nevertheless, in Section~\ref{S: rule out inf dens} we prove that the density is actually finite at any point in $\Gamma$ where $E$ is touched from the interior by a $C^1$ set $G$. The main idea is that, if such a point has infinite density, then a suitable mass-rescaling of $E$ would provide a varifold supported in a wedge, which is stationary with respect to vector fields in $\TVF$. This comes from the fact that the touching set $G$ imposes a multiplicity bound on the wetting region $\partial E \cap \partial H$ of $E$. Since the varifold touches the spine of the wedge, we construct an appropriate vector field in $\TVF$ that violates \eqref{eqn: aniso isotropic stationary}. This guarantees the existence of blow-ups, though not necessarily with a conical structure. In Section~\ref{S: BU union of planes}, we prove the existence of a blow-up that is finite union of half-hyperplanes by refining a non-linear maximum principle argument.

\noindent \textbf{Acknowledgments.} Antonio De Rosa was funded by the European Union: the European Research Council (ERC), through StG ``ANGEVA'', project number: 101076411. Views and opinions expressed are however those of the authors only and do not necessarily reflect those of the European Union or the European Research Council. Neither the European Union nor the granting authority can be held responsible for them.
Robin Neumayer is supported by NSF grants DMS-2340195, DMS-2155054, and DMS-2342349.

\section{Preliminaries}\label{sec: prelimininaries}

In this section, we fix notation, definitions, and background that will be used throughout this article. 
The open ball of radius $r>0$ and center $x$ in $\R^{n+1}$ is denoted by $B_r(x)$, and, if $x=0$, we simply write $B_r$. The sphere of radius $1$ is then $\bbS^n:= \partial B_1$.  For a vector $\nu\in\bbS^n$, we denote by $\nu^\perp$ its orthogonal space.

\subsection{Rectifiable sets and sets of finite perimeter.}
For any $k\in \mathbb{N}$, the Hausdorff measure of dimension $k$ in $\R^{n+1}$ is denoted by $\cH^k$ and the $\cH^k$-measure of the $k$-dimensional unit ball is denoted by $\omega_k$. 
For any vector-valued Radon measure $\mu$ on $\R^{n+1}$, we denote with $|\mu |$ the total variation of $\mu$.
A Borel set $M$ is $n$-rectifiable if  $M$ can be covered, up to a $\cH^n$-null set, by countably many images of Lipschitz images of $\R^n$. 
 A Lebesgue measurable set $E$ is a set of locally finite perimeter if for every compact set $K \subset {\R^{n+1}}$ we have
\[
\sup \left\{\int_E \div X(x) \rmd x\ : \ X \in C_{\mathrm{c}}^1({K}, \R^{n+1}), \sup_{{K}}|X| \leq 1\right\}<\infty .
\]
If this quantity is bounded independently of $K$, we say that $E$ is a set of finite perimeter. By the Riesz representation theorem, for any set of locally finite perimeter $E$, there is a vector-valued Radon measure $\mu_E$, called the Gauss-Green measure of $E$, such that $\int_E \text{div}T \, dx = \int T\cdot d \,\mu_E$ for any $T \in C^1_c({\R^{n+1}}, \R^{n+1}).$  The {(relative)} perimeter of $E$ in $A\subset {\R^{n+1}}$ is $P(E;A) := |\mu_E|(A)$.

The {reduced boundary $\partial^* E$ of $E$} and the {outer unit normal $\nu_E(x)$ to $E$ at $x$} are then defined as follows:
\[ 
    \partial^* E := \left \{x\in\spt(\mu_E): \nu_E(x) := \lim_{r\to 0}\frac{\mu_E(B_r(x))}{|\mu_E|(B_r(x))}\mbox{ exists and is a unitary vector} \right\}.
\]
For a set $E$ of locally finite perimeter,
$\partial^* E$ is $n$-rectifiable and $\mu_E = \nu_E\cH^n\res\partial^*E$; see \cite[Remark 15.3 and Corollary 16.1]{maggi2012sets}.

\begin{remark}\label{r:openrep}
    As stated in the introduction, we tacitly choose a representative for any set of finite perimeter $E$ in $\R^{n+1}$ with $\partial E = \overline{\partial^*E}$; see \cite[Prop. 12.19]{maggi2012sets}.
If $|\overline{\partial^*E}| =0$, we may, and always do, choose this representative to be an open set. Indeed, by \cite[Prop. 12.19]{maggi2012sets} and its proof,  any set of finite perimeter $E$ in $\R^{n+1}$ has
 $\overline{\partial^*E}  = \{x \in \R^{n+1} : 0< |E\cap B_r(x) |< \omega_n r^n,\  \forall r>0\}$, and the open set $A_1 = \{x \in \R^{n+1} : |B_r(x)\cap E| = \omega_n r^n \text{ for some } r>0\} $ has $\partial A_1 \subset \overline{\partial^*E}$ and $|E\Delta A_1| \leq | \overline{\partial^*E}|$. Thus, if $ | \overline{\partial^*E}|=0$, then $A_1$ is the desired representative of $E$.
Observe that if $\mathcal{H}^n({\partial E} \setminus\partial^*E)=0$, i.e. if  \eqref{A:RegAnis:Intro} holds, then in particular $ | \overline{\partial^*E}|=0$ and we have such a representative.
\end{remark}

A sequence of sets of locally finite perimeter $\{E_k\}$ converges in $L^1_{loc}$ to a set of locally finite perimeter $E$ if $|(E_k\Delta E)\cap K|\to0 $ for every compact set $K$.

\subsection{Elliptic integrands}\label{ssec: prelimm aniso}
We fix an elliptic integrand $\Phi :\R^{n+1} \to \R_+$, that is, an even, positively one-homogeneous function  $\Phi \in C^{2,\alpha}(\R^{n+1}\setminus \{0\})$ satisfying the uniform ellipticity condition
\begin{equation}
    \label{eqn: ellipticity}
    D^2\Phi_\nu[\tau, \tau] \geq \gamma \frac{ |\tau|^2 - \big( \tau\cdot \frac{\nu}{|\nu|}\big)^2}{ |\nu|}\mbox{ for any }\nu,\tau\in \R^{n+1}, \nu\neq 0,  \mbox{ and some }\gamma > 0.
\end{equation}
We let $m_\Phi := \inf\{\Phi(\nu):\nu\in\bbS^n\}$ and $M_\Phi := \sup\{\Phi(\nu):\nu\in\bbS^n\}$. 
Let $\Phi^*$ be the dual norm to $\Phi$, i.e. 
\begin{equation}\label{E:def anis dist and dual norm}
	\Phi^\ast(y) = \sup\{ y\cdot x : x\in \R^{n+1}\mbox{ and }\Phi(x) \leq 1\}\,.
\end{equation}
By duality, the Fenchel inequality $x\cdot y \leq \Phi^\ast(x)\Phi(y)$ holds for any $x,y\in \R^{n+1}$. 
By one-homogeneity, $\Phi(z) = z\cdot\nabla\Phi(z)$ and $\nabla \Phi$ and $\nabla\Phi^\ast$ are zero-homogeneous.
The {Wulff shape of radius $r>0$ centered at $y$} is 
\begin{equation*}
    \cW_r(y) := \left\{ x\in\R^{n+1}: \Phi^\ast(y-x) < r\right\}.
\end{equation*}
For a set of finite perimeter $E$, we let $\nu_E^\Phi(x) = \nabla \Phi \circ \nu_E(x)$ be the Cahn-Hoffman normal.
Observe that 
\begin{equation}
    \label{eqn: CH map} 
    \na \Phi(\nu_\cW (x)) = x  \text{ for each }x \in \partial \cW \qquad \text{ and }\na \Phi^*(x) = \nu_\cW(\tfrac{x}{\Phi^*(x)}).
\end{equation}
In particular,  $\Phi^\ast\circ\nabla\Phi \equiv 1$ and $\nabla\Phi\circ \nabla\Phi^\ast \equiv \mathrm{id}$ for $x \in \partial \mathcal{W}_1$.

Fix $e \in \bbS^n \cap \partial H$.  For $\nu_1, \nu_2 \in \bbS^n$ in the half $2$-plane $\textup{span}\{e_{n+1}, e_+\}= \{a e_{n+1} + b e : a \in\R, b>0 \}$, we have
    \begin{equation}\label{I:contrapositive ordering Cahn-Hoffman}
    \nu_1 \cdot \nu_H > \nu_2 \cdot \nu_H \quad \iff \quad \nabla \Phi (\nu_1) \cdot \nu_H >\nabla \Phi(\nu_2) \cdot \nu_H.
    \end{equation}
This is shown in  \cite[Proposition 3.1]{jia2023alexandrov}. 

As a consequence of uniform ellipticity and one-homogeneity, there is a constant $c>0$ such that 
    \begin{equation}\label{L: Unif Elliptic}
        \Phi(x)\Phi(y) - \left(\nabla \Phi(x)\cdot y\right)\left(\nabla \Phi(y)\cdot x\right) \geq c |x-y|^2, \mbox{ for any }x,y.
    \end{equation}

If $\Omega$ is an open set with $C^2$ boundary, then the (scalar) anisotropic mean curvature $H_\Omega^\Phi (x)$ is given by 
\[
H_\Omega^\Phi  (x) = \text{div}^{\partial \Omega} ( \nu_\Omega^\Phi(x) ) =  \text{div}^{\partial \Omega} (\nabla \Phi(\nu_\Omega(x)), \mbox{ for }x \in \partial \Omega.
\]
The eigenvalues of $D\nabla \Phi(\nu_\Omega(x))$ are the anisotropic principle curvatures, denoted by $\kappa_i^\Phi(x)$; clearly $H_\Omega^\Phi  (x) = \sum\kappa_i^\Phi(x)$. Another equivalent point-wise expression for $H_\Omega^\Phi  $ is given by 
\begin{equation}\label{E:MC 0-level set}
H_\Omega^\Phi  (x)=   \tfrac{1}{\Phi(\nu_\Omega(x))}   B_\Phi(\nu_\Omega(x)):D \nu_\Omega^\Phi(x), \mbox{ for }x \in \partial \Omega.
\end{equation}
Here $A:B := \mathrm{tr}(A^tB)$ denotes the Hilbert–Schmidt inner product, $B_\Phi(\nu) := \Phi(\nu) \text{Id}  -\nu \otimes \nabla\Phi(\nu)$ for $\nu \in \bbS^n$ is as defined above, and $D\nu_\Omega^\Phi(x)$ denotes the derivative in $\R^{n+1}$ of the function $\nu_\Omega^\Phi(x)$ extended constantly along normal rays to be defined in a tubular neighborhood of $\partial \Omega$.
Indeed, to see \eqref{E:MC 0-level set}, fix $x \in \partial \Omega$ and an orthonormal basis $\{\tau_1 ,\dots, \tau_n\}$ for $T_x\partial \Omega$ that diagonalizes the second fundamental form $A_{\Omega}(x)$ of $\partial \Omega$ at $x,$ i.e. $A_\Omega(x) = \sum_{j=1}^n\kappa_j \tau_j \otimes \tau_j$, and let $\nu = \nu_\Omega(x)$ so that $\{\tau_1,\dots, \tau_n, \nu\}$ is an orthonormal basis for $\R^{n+1}$.
Since $D^2\Phi_\nu = \partial_{ij}\Phi (\nu) \, \tau_i \otimes \tau_j$ by homogeneity, in this basis we have $D\nu_\Omega^\Phi(x)  = \sum_{i,j=1}^n \kappa_j\partial_{ij}\Phi(\nu) \tau_i \otimes \tau_j.$   
So, we have $H_\Omega^\Phi  (x)  = \text{trace}(D\nu_\Omega^\Phi(x)) = \sum_{i=1}^n \kappa_i\partial_{ii} \Phi(\nu)$, while, since $D\nu_\Omega^\Phi(x)$ is purely tangential, i.e. $\nu_\Omega(x) \in \ker D\nu_\Omega^\Phi(x)$, then $B_\Phi(\nu_\Omega(x)) : D\nu_\Omega^\Phi(x) $ is simply equal to $ \Phi(\nu_\Omega(x)) \text{Id} :{D\nu_\Omega^\Phi(x)} =\sum_{i=1}^n \kappa_i\partial_{ii} \Phi(\nu).$ This shows \eqref{E:MC 0-level set}.

\subsection{Varifolds} \label{subsec: prelim varifolds}
An $n$-varifold on $\R^{n+1}$ is a positive Radon measure $V$ on $\R^{n+1} \times \bbS^n$ such that $V(A \times S) = V(A\times (-S))$ for any $A\subset \R^{n+1}$ and $S\subset \bbS^n$. The weight $\|V\|$ of $V$ is the Radon measure defined by
\begin{equation*} 
    \|V\|(A) := V(A\times \bbS^n), \quad \forall A\subset \R^{n+1}.
\end{equation*}
{An $n$-varifold $V$ is {$n$-rectifiable} if there exist an $n$-rectifiable set $M$ and a Borel function $\theta:\R^{n+1}\to \R_+$ such that $V = \theta\cH^n\res (M\cap U) \otimes (\frac12\delta_{\nu_M(x)} + \frac12\delta_{-\nu_M(x)})$ where $\delta_a$ denotes the Dirac delta concentrated at $a$.
To any 
 $n$-rectifiable set $M$ in $\R^{n+1}$, we can associate the $n$-rectifiable varifold $\var(M) : = \cH^n\res M \otimes (\frac12\delta_{\nu_M(x)} + \frac12\delta_{-\nu_M(x)})$.}

For a diffeomorphism $\psi \in C^1_c(\R^{n+1},\R^{n+1})$ and an $n$-varifold $V$, the {push-forward $\psi^\# V$ of $V$  with respect to $\psi$} is the $n$-varifold defined by 

\begin{equation}\label{E:pushforward-varifold}
    \int_{\R^{n+1}\times \bbS^n}f(x,\nu)d(\psi^\#V)(x,\nu) = \int_{{\R^{n+1}\times \bbS^n}}f\Big(\psi(x),\tfrac{D\psi_x(\nu)}{|D\psi_x(\nu)|}\Big)J\psi(x,\nu^\perp) \, dV(x,\nu),
\end{equation}
for all $f\in C^0_c(U \times \bbS^n)$. Here $D\psi_x$ is the differential of $\psi$ and $J\psi(x,\nu^\perp)$ is the Jacobian determinant of the differential $D\psi_x$ restricted to $\nu^\perp$, i.e., 
\[
J\psi(x,\nu^\perp)
:=\sqrt{\det\Big(\big(D\psi_x\big|_{\nu^\perp}\big)^* D\psi_x\big|_{\nu^\perp}\Big)}.
\]
The upper and lower $n$-dimensional density of a Radon measure $\mu$ on $\R^{n+1}$ at $x$ are,  respectively,
\begin{equation*} 
    \udens{n}{\mu}{x} := \limsup_{r\to 0^+} \frac{\mu(B_r(x))}{\omega_n r^n}\qquad \mbox{ and }\qquad \ldens{n}{\mu}{x} := \liminf_{r\to 0^+} \frac{\mu(B_r(x))}{\omega_n r^n}.
\end{equation*}
If $\udens{n}{\mu}{x} = \ldens{n}{\mu}{x},$ we call this number the {density of $\mu$ at $x$} and denote it by $\dens{n}{\mu}{x}$. { For an $n$-varifold, we let $\udens{n}{V}{x}:=\udens{n}{\|V\|}{x}$  and for an $n$-rectifiable set $M$, we let $\udens{n}{M}{x}:=\udens{n}{\cH^n\res M}{x}$ (and likewise for the lower density and the density).  }

Given $x\in\spt(\|V\|)$ and $r>0$, let
\begin{equation}
    \label{eqn: blowup seq}
    V_{x, r}:= \left(\iota_{x, r}\right)^{\#} V,
\end{equation}
and let $\cB_x(V)$ denote the set of blow-ups of $V$ at $x$, i.e. the set of all varifolds $V'$ arising as subsequential limits of $V_{x, r}$ as $r \to 0$. 
Here convergence is in the sense of varifolds, i.e. as Radon measures.
In general, $\cB_x(V)$ may be empty. However, if $\Theta^*(V,x)<+\infty$, then any blowup sequence is pre-compact and thus $\cB_x(V)$ is nonempty in this case.
\noindent 

\subsection{First variation and capillary varifolds}\label{ssec: prelim first variation} 
The anisotropic surface energy naturally extends to $n$-varifolds by letting $\bPhi(V) = \int_{\R^{n+1} \times \bbS^n} \Phi(\nu) \, \rmd V(x,\nu)$. 
The first variation of an $n$-varifold $V$ with respect to the energy $\bPhi$ is the order one distribuition defined by $\Fv V(X):=\frac{d}{dt}|_{t=0} \bPhi( (\text{Id} +tX)^\# V) $ for a vector field $X \in C^1_c(\R^{n+1}).$ A classical computation {(cf. \cite[Lemma A.2]{de2018rectifiability})} shows that 
\begin{equation}\label{E:first-variation-Ani0}
\begin{aligned}
	\Fv V(X) &= \int B_\Phi(\nu): DX(x) \rmd V(x,\nu) .
\end{aligned}
\end{equation}
{Recall that $B_\Phi(\nu) := \Phi(\nu) \text{Id}  -\nu \otimes \nabla\Phi(\nu)$.}

For a set $E\subset H$ that satisfies \eqref{eqn: aniso isotropic stationary}, it is not useful to consider blow-ups of $E$ as a set of finite perimeter or blow-ups  of the varifold $\var(\partial^*E)$, because the stationarity condition \eqref{eqn: aniso isotropic stationary} is not preserved in either blow-up limit, as an interior component could collapse on $\partial H$ in the limit. This issue has already been recognized, for instance, in \cite{nick2024regularity,wang2024allard}. Instead, following \cite{nick2024regularity}, we associate to $E$ the  \textit{capillary varifold} $V$ on $\R^{n+1}$ with $\spt\| V\| \subset \overline{H}$ defined by 
\begin{equation}\label{eqn: weighted varifold}
V = \var(\partial^*E \cap H)  - \frac{\omega}{\Phi(\nu_H)} \var(\partial^*E \cap \partial H).
\end{equation}
$V$ is indeed a well-defined varifold in the case $\omega \leq 0$. A direct computation shows that $E$ satisfies  \eqref{eqn: aniso isotropic stationary} if and only if $V$ satisfies
\begin{equation}\label{eqn:fv}
\Fv V(X)  - \lambda \int X(x) \cdot \nu \, \rmd  V(x,\nu) = 0 \qquad \mbox{ for any } X \in \TVF.
\end{equation}
It is this stationarity condition that passes to the blow-up limit. To be more precise, for a point $p \in \Gamma$, any blow-up $V_0 \in \cB_p(V)$ -- should it exist -- satisfies $
\Fv V_0(X)=0$  for any $X \in \TVF$.

\section{Finiteness of the upper density} 
\label{S: rule out inf dens} Fix $\omega \in (-\Phi(\nu_H) , 0]$ and 
let $E$ be  a critical point of $\bPhi_\omega$ among volume preserving variations, so that $E$ satisfies \eqref{eqn: aniso isotropic stationary} for some $\lambda >0.$
The proof of Theorem~\ref{thm: stationary implies viscosity} in Section~\ref{sec: stat to visc} is based on a blowup argument at points on the capillary boundary $\Gamma$.

In general, however, no such blow-up may exist. For the area functional and {the capillary varifold} $V$ defined as in \eqref{eqn: weighted varifold}, a blow-up $V_0$ exists and its support is a cone, see Lemma~\ref{lemma: cone} below, but even in this case, no further (useful) structural information can be deduced about $V_0$.
The key idea is that, in order to prove Theorem~\ref{thm: stationary implies viscosity}, we only need to analyze blow-ups of the varifold $V$ at those points $p \in \Gamma$ where $E$ is touched from the interior by a $C^1$ set $G$ with $\nu_G^\Phi(p) \cdot \nu_H >-\omega$.

This touching set imposes additional regularity of $\partial E$ at $p$ and in particular yields a blow-up of the {capillary} varifold $V$ with substantially more structure than merely being a cone (even for the anisotropic case). More precisely, at such a point $p$, we prove in Theorem~\ref{T: finiteness density} below that $\Theta^*_n(V,p) <+\infty$, which in particular guarantees that the blowup set $\cB_p(V)$ is nonempty.  Later, in Proposition~\ref{P: flat BU under fin dens}, we show that there is a blowup $V_0 \in \cB_p(V)$ whose support {in $H$} is the union of finitely many half-planes.

\begin{theorem}
\label{T: finiteness density}
Let $\omega \in (-\Phi(\nu_H), 0]$ and let {$E\subset H$ be a set of finite perimeter satisfying \eqref{eqn: aniso isotropic stationary}}. For $p \in \Gamma$, suppose there is an open $C^1$ set $G\subset \R^{n+1}$ touching $E$ from the interior at $p$ with $\nu_{G}^\Phi(p) \cdot  \nu_H > 0$. Then
\begin{equation*}
	\Theta^*_n(\partial^* E, p) < +\infty.
\end{equation*}   
Equivalently, $\Theta^*_n(V,p ) < +\infty$  for the {capillary} varifold $V$ associated to $E$  in  \eqref{eqn: weighted varifold}.
\end{theorem}

\begin{remark}
    In the case of the area functional, the conclusion of Theorem~\ref{T: finiteness density} is an immediate consequence of the classical monotonicity formula, which holds for any $p \in \partial H$; see the proof of Lemma~\ref{lemma: cone}.
\end{remark}
Before starting the proof of Theorem~\ref{T: finiteness density}, we need some notation. Let $\cW := \cW_1$ and let $K := \cW\cap \partial H$ be the intersection of the Wulff shape with $\partial H$. Since $\cW$ is open, uniformly convex, and contains the origin, $K$ is relatively open in $\partial H$ and uniformly convex. 
Fix a direction $e \in \bbS^{n} \cap \partial H$ and let $x_e$ be the unique point in $\partial K$ such that $\nu_K(x_e) = e$. Here, $\partial K$ denotes the boundary relative to $\partial H$.
 Thus,  we have
 \begin{equation}\label{E: normal to W at x_e}
     \nu_{\cW}(x_e)\in\mathrm{span}\{e_+, e_{n+1}\} := \{ x\in\R^{n+1}: x = ae + be_{n+1}\mbox{ for }a>0, b\in \R\}\,.
 \end{equation}
Correspondingly, define the closed wedge 
\[
W(\Phi, e) := \overline{H} \cap \{ x \in \R^{n+1} : x\cdot \nu_{\cW}(x_e)  \geq 0\}. 
\]
This wedge is invariant along $\text{span}\{ e, e_{n+1}\}^{\perp}$, and $W(\Phi, e)+  x_{e}$ is contained in the complement of $\cW.$

A set $W \subset \R^{n+1}$ is called {\it strict subwedge} of $W(\Phi, e)$ if it takes the form 
\begin{equation}\label{eqn: strict subwedge}
    W= \overline{H} \cap \{ x \in \R^{n+1} : x \cdot \nu \geq 0 \}\qquad \mbox{ for }\nu \in \text{span}\{e_+,e_{n+1}\}\mbox{ with }\nu \cdot \nu_H > \nu_{\cW}(x_e) \cdot \nu_H.
\end{equation}
Note that $W$ is a strict subset of $W(\Phi, e)$ and is a closed wedge that is invariant along $\text{span}\{e, e_{n+1}\}^{\perp}$.
The most relevant strict subwedges arise as follows.

\begin{lemma}
    \label{lem: wedge containment} 
Let $G\subset\R^{n+1}$ be an open set with $C^1$ boundary. Suppose $p \in \partial G\cap \partial H$ is a point  such that
 $\nu_G^\Phi(p)\cdot \nu_H >0$. Then either $\nu_G(p) =\nu_H$ or  $\overline{H}  \setminus T_p G$ is a strict subwedge of $W(\Phi,e)$, where $e$ is the unique vector in $\bbS^n\cap \partial H$ with  $
\nu_G(p) \in \text{span}\{e_+,e_{n+1}\}$.
\end{lemma}
\begin{proof}
Suppose $\nu_G(p) \neq\nu_H$. Since  $\overline{H}  \setminus T_p G =\overline {H} \cap \{ x \in \R^{n+1} : x\cdot \nu_G(p) \geq 0\}$,  we need only to show  that
\begin{equation}\label{eqn: ordered inner prod 1}
    \nu_G(p) \cdot \nu_H > \nu_\cW(x_e) \cdot \nu_H.
\end{equation}
By assumption, $\nabla \Phi(\nu_G(p))\cdot \nu_H >  0$. On the other hand, $\nabla\Phi( \nu_{\cW}(x_e))\cdot \nu_H = 0$ since, by \eqref{eqn: CH map}, $\nabla\Phi( \nu_{\cW}(x_e)) = x_e \in \partial H$. So, 
\begin{equation}
    \label{eqn: CH strict}
    \nabla \Phi(\nu_G(p))\cdot \nu_H  > \nabla\Phi( \nu_{\cW}(x_e))\cdot \nu_H.
\end{equation}
Since $\nu_G(p)$ and $\nu_\cW(x_e)$ lie in $\text{span}\{e_+, e_{n+1}\}$, \eqref{eqn: ordered inner prod 1} follows from \eqref{eqn: CH strict} and \eqref{I:contrapositive ordering Cahn-Hoffman}.
\end{proof}

\begin{proof}[Proof of Theorem~\ref{T: finiteness density}]
If $\nu_G(p) = \nu_H$, then $p \in \partial^*E$. Thus $\Theta^n(\partial^*E, p)$ exists and is equal to 1; c.f. \cite[Corollary 15.8]{maggi2012sets}.  

Now  assume $\nu_G(p) \neq \nu_H.$ 
Let $e$ be the unique vector in $\bbS^n\cap \partial H$ such that $
\nu_G(p) \in \text{span}\{e_+,e_{n+1}\}$.
By Lemma~\ref{lem: wedge containment}, the closed wedge $ \overline{H}\setminus T_p G$ is a strict subwedge of $ W(\Phi, e)$. Let $W$ be a strict subwedge of $W(\Phi, e)$ that contains $ \overline{H}\setminus T_p G$ as a strict subset. Let $\bar{\rho}\in (0,1)$ be a fixed constant depending only on $W$, $\Phi$, and $n$ to be specified below.

Assume by way of contradiction that $\Theta^*_n(V, p) = +\infty$ for the varifold $V$ defined in \eqref{eqn: weighted varifold}. So, we may find a sequence of scales $\{r_k\}_k$ with $r_k \to 0$ such that the varifolds $V_{k} = \iota_{p, r_k}^\# V$ satisfy $\| V_{k}\|({B_{\bar{\rho}/2}}) \to +\infty.$ 
Consider the area blowup set  $Z$ associated to the sequence $\|V_k\|$ defined by
\[
Z = \{ x\in \R^{n+1}: \limsup_k \|V_k\|(B_r(x)) = +\infty \text{ for all } r >0\}.
\]
From our choice of $W$, $\spt(\|{V}_{k}\|) \cap B_1$ is contained in $W \cup  \partial H$ for $k$ large enough. Since  $\|V_{k}\|\llcorner \partial H \leq \mathcal{H}^n \llcorner \partial H$, this means that any $x \in Z\cap \overline{B}_{\bar{\rho}/2}$ has $\limsup_k \|V_k\|(B_{r}(x) \cap W) = +\infty$ for each $r >0$.  By \cite[Lemma 4.2]{philippis2019area} and the assumption $\| V_{k}\|({B_{\bar{\rho}/2}}) \to +\infty$, the intersection $Z\cap \overline{B}_{\bar{\rho}/2}$ is nonempty. Thus by set containment, 
\begin{equation}
    \label{eqn: main consequence of area blowup}
    \limsup_k \|V_k\|(B_{\bar\rho}\cap W) = +\infty.
\end{equation}

\underline{Claim:} We claim there are constants $\bar{c}, \bar{\rho}>0$ and a vector field $Y \in \TVF$ supported in $B_1$,  depending only on $W$, $\Phi$ and $n$, such that $|Y(x)|\geq \bar{c}$ on $B_{\bar\rho}\cap W$ and, letting  $\Omega:=\spt(Y)\cap W$, 
\begin{align}\label{eqn: bound in claim}
B_\Phi(\nu) : DY(x) \leq - 2\bar{c}\, |Y(x)| \qquad \text{ for all }(x,\nu) \in \Omega \times \bbS^n\,.
\end{align}

\medskip

Assuming for now that the claim holds, let us see how it allows us to complete the proof of the theorem. On one hand, by scaling, we know that 
\begin{equation}\label{eqn: first var zero}
\alpha_{k} :=  \delta_\Phi V_k(Y) - \lambda r_k \int Y\cdot \nu \, dV_k\,=0 .
\end{equation}
On the other hand, for all $k$ large enough, $\spt(\|V_k\|)\cap\spt(Y) \subset \Omega \cup D$ where $D = \partial H \cap \{ x\cdot e<0 \}\cap B_1$. So, directly writing the first variation term and splitting the domains of integration shows 
\begin{align*}
 \alpha_k &= \int_{\Omega \times \bbS^n} \big(  B_\Phi(\nu) : DY - \lambda r_k\, Y\cdot \nu \big) \, dV_k + \int_{D \times \bbS^n}\big(  B_\Phi(\nu) : DY -  \lambda r_k \, Y\cdot \nu\big)\, dV_k= \alpha_k^\Omega +\alpha_k^D\,.
 \end{align*}
By the claim above, we have 
 \[
 \alpha_k^{\Omega} \leq \int (-2\bar{c} + \lambda r_k )|Y| \, dV_k
  \leq  -\bar{c}\int |Y| \, dV_k
 \leq -\bar{c}^2\,  \|V_k \|(B_{\bar{\rho}}\cap W) ,
 \]
 where the second inequality holds for $k$ large enough.
On the other hand, by construction, $\|V_k\|(D) \leq \frac{\omega}{\Phi(\nu_H)} \mathcal{H}^n(D) \leq \omega_n/2$ and hence $\alpha_k^D \leq C $ for a constant  $C= C(\Phi, n,\|Y\|_{C^1}).$ So, by \eqref{eqn: main consequence of area blowup}, we have 
\[
\liminf_k \alpha_k =\liminf_k ( \alpha_k^\Omega + \alpha_k^D )\leq -\bar{c}^{{2}}\,  \limsup_k\| V_k\| (B_{\bar{\rho}}\cap W) + C  = -\infty,
\]
contradicting \eqref{eqn: first var zero}.
\end{proof}

 The rest of the section is dedicated to proving the claim inside the previous proof. The construction is inspired by the proof of Solomon-White's maximum principle \cite{solomon1989strong} for varifolds that are stationary with respect to anisotropic functionals, an argument that was later adapted to the context of $(n,h)$-sets in \cite{philippis2019area}. 
The main new challenge in the present setting is to construct the vector field $Y$ to be tangential to $\partial H$.
The first step is to construct a suitable auxiliary function $f$ with $\nabla\Phi(\nabla f)$ tangential to $\partial H$.

\begin{lemma}\label{L:fnc f}
Fix $e \in \bbS^{n+1}\cap \partial H$. There exist $\Lambda>0$ and $r_1\in(0,1)$ depending on $n,\Phi,$ and $e$  such that the function $f= f_e : \R^{n+1} \to \R$ defined by 
$$f(x) := \Phi^\ast(x + x_e) - 1 - \Lambda x_{n+1}^2
$$
has the following properties:
\begin{enumerate}[\upshape (i)]
    \item $f(0)=0$ and $\nabla f(0)\in \mathrm{span}(e_+, e_{n+1})$,
    
    \item $\nabla\Phi (\nabla f(x)) = (x+x_e)/\Phi^*(x+x_e)\in \partial H$ for any $x\in \partial H$.

    \item $|\nabla f(x) | \geq \tfrac{1}{2 M_{\Phi}}>0$ for $x \in B_{r_1}$, so the level sets $\{ f=a\} \cap B_{r_1}$ are $C^{2,\alpha}$, and 
    $$H^\Phi_{\{f < a\}} (x) \leq -1$$
    for each $x \in \{f=a\} \cap B_{r_1}$.

    \item For any strict subwedge $W$ of $W(\Phi, e)$, there exists $\bar r \in (0, r_1]$ such that 
\[
W\cap B_{\bar{r}}\subset \{f \geq 0\}\cap B_{\bar{r}} \quad \text{ and } \quad \{f=0\}\cap W \cap B_{\bar{r}} = \{0\}.
\]
\end{enumerate}
\end{lemma}

\begin{proof}
The first half of (i) holds since $x_e \in \partial \cW$, and the second half follows from \eqref{E: normal to W at x_e} since $\nabla f(0) = \nabla \Phi^\ast(x_e) = |\nabla \Phi^*(x_e)| \nu_\cW(x_e)$. Similarly, (ii) follows from \eqref{eqn: CH map} since $\nabla f(x) = \nabla\Phi^\ast(x+x_e)$ for any $x\in\partial H$.

Next we show (iii). We have $|\nabla f(0)|\geq 1/M_\Phi$ since $\Phi(\nabla f(0)) = 1$ so the lower bound on $|\nabla f(x)|$ follows from continuity of $\nabla f$ by choosing $r_1$ sufficiently small, and the regularity of the level sets of $f$ then follows from the implicit function theorem and  $C^{2,\alpha}$ regularity of $\Phi^*$. The same regularity guarantees that the anisotropic mean curvature of level sets will vary continuously, so the upper bound on the anisotropic mean curvature will follow by choosing $r_1$ small provided we show $H^\Phi_{\{f<0\}}(0) \leq -2$. 

We show this by direct computation as follows, letting $\Omega := \{f<0\}$. The outer unit normal to $\Om$ at $p$ is given by $\nu_{\Omega}(p) = \nabla f(p)/|\nabla f(p)|$. Thus by homogeneity, the Cahn-Hoffman normal is $\nu_{\Omega}^\Phi(p) = \nabla\Phi(\nabla f(p))$ and, since $\nabla f(p) = \nabla \Phi^\ast(p+x_e) - 2\Lambda p_{n+1}e_{n+1}$ the anisotropic mean curvature (defined in Section~\ref{ssec: prelim first variation}) is given by
\[ H^\Phi_{\Omega} (p)  = \text{trace} (D^2 \Phi_{\nabla f(p)} \, D^2\Phi^\ast _{p+x_e} )- 2\Lambda D^2\Phi_{\nabla f(p)}[e_{n+1}, e_{n+1}]  = I(p) - II(p). \]
Tangentially differentiating the identity $\nabla\Phi\circ\nabla\Phi^\ast(y) =y$ for $y \in \partial \mathcal{W}$ yields $I(0) = n$.
For $II(0)$, again using that $\nabla f(0) = |\nabla \Phi^*(x_e)| \,\nu_{\cW}(x_e) \in \text{span}(e_+, e_{n+1})$. Writing $\nu_{\cW}(x_e) = a e + b e_{n+1}$ for $a>0$, the uniform ellipticity of $\Phi$  \eqref{eqn: ellipticity} guarantees that 
\[
 II(0) = 2\Lambda
D^2\Phi_{\nabla \Phi^\ast(x_e)}[e_{n+1},e_{n+1}] \geq 2\Lambda\,\frac{\gamma}{|\nabla \Phi^\ast(x_e)|}  a^2 >0\,.
\] 
Thus $H^\Phi_{\Omega}(0)= n-2\Lambda\,\gamma a^2/|\nabla \Phi^\ast(x_e)|$, which is at most $-2$ provided $\Lambda \geq (n-2)|\nabla \Phi^*(x_e)| /2\gamma a^2$. Hence, as observed above, item (iii) follows by continuity.

We now show (iv). Let $W$ be a strict subwedge of $W(\Phi, e)$. 
Let $\pi$ denote the projection onto $\text{span}(e,e_{n+1})$ and for $x \in W,$ write $x = \pi x + x^\perp$. Recall that, since $\Phi$ is $C^2$, $\Phi^*$ satisfies the ellipticity condition \eqref{eqn: ellipticity} for some $\gamma^*>0$. Moreover, $x \cdot \nu_\cW (x_e ) =\pi x \cdot \nu_\cW (x_e ) \geq 0$, with the inequality holding since $W\subset W(\Phi, e)$. Next, a Taylor expansion at $x=0$ shows 
\begin{align*}
    f(x)  \geq |\nabla \Phi^\ast(x_e)| \, \nu_\cW(x_e) \cdot \pi x& + \frac{1}{2} D^2 \Phi^*_{x_e}[x,x] - \Lambda x_{n+1}^2 + o(|x|^2)\\
    & \geq \frac{\gamma^*}{|x_e|} |x^\perp|^2 - \Lambda |\pi x|^2 + o(|x|^2).
\end{align*}
We consider two cases.

\textit{Case 1:} For  $x \in W \cap B_{\bar{r}}$ such that $|\pi x|^2 \leq \min\{ 1, \tfrac{\gamma^* }{2\Lambda|x_e|} \} |x^\perp|^2$, we have 
\[
f(x) \geq \frac{\gamma^*}{2|x_e|} |x^\perp|^2 + o(|x|^2) \geq \frac{\gamma^*}{4|x_e|} |x|^2 + o(|x|^2)   \geq \frac{\gamma^*}{8|x_e|} |x|^2 ,
\]
where the final inequality holds provided we choose $\bar{r}$ small enough depending only on $\Phi$ (via $\gamma^*$ and $|x_e|$). 

\textit{Case 2:} Notice first that, by definition---recall \eqref{eqn: strict subwedge}---the hyperplane 
$\nu_{\cW}(x_e)^\perp$ intersects $W\cap \text{span} (e, e_{n+1})$ only at the origin. So, by compactness, there is a positive constant $\alpha$, depending on $W$ and $\Phi$, such that $ \nu_{\cW}(x_e)\cdot \pi x \geq \alpha |\pi x|$ for each $x \in W$. 
So, the Taylor expansion above yields
\[
f(x) \geq |\nabla \Phi^\ast(x_e)| \, \nu_\cW(x_e) \cdot \pi x + o(|x|) \geq \alpha |\nabla \Phi^\ast(x_e)|\, |\pi x| +  o(|x|).
\]
So, for $x \in W \cap B_{\bar{r}}$ such that $|\pi x|^2 \geq  \min\{ 1, \tfrac{\gamma^* }{2\Lambda|x_e|} \} |x^\perp|^2$, we have $|x| \leq C|\pi x|$ for $C= 1/(1 + \min\{ 1, \tfrac{\gamma^* }{2\Lambda|x_e|} \})$. So, we may choose $\bar{r}$ small enough in terms of $\alpha$ and $\Phi^*$ so that we may absorb term $o(|x|)$ to find $f(x) \geq c|x|$ in this case.
This completes the proof. 
\end{proof}

We are now ready to prove the claim. 

\begin{proof}[Proof of Claim]
Recall that $e \in \bbS^n \cap \partial H$ is fixed and $W$ is a strict subwedge of $W(\Phi, e)$.

\noindent{\it Step 1: The vector field $Y$.} Let $f=f_e$ and $\bar{r}\in(0,1)$ be as in Lemma~\ref{L:fnc f}. 
By  Lemma~\ref{L:fnc f}(iv) and continuity of $f$, we may find $\eta_1>0$
such that 
\begin{equation}\label{eqn: f ge eta}
f(x) > \eta_1 \mbox{ for any }x\in W\cap (B_{\bar{r}} \setminus B_{\bar{r}/2}).
\end{equation} 
For a fixed number $\bar{\eta}\in (0,\eta_1)$  to be specified later in the proof, define 
\begin{equation}
    \label{eqn: def eta}
\eta(t) := \begin{cases}
    0 & \text{ for }t\geq \bar{\eta},\\
    \bar{\eta}-t  & \text{ for } t\leq \bar{\eta}.
\end{cases}
\end{equation}
Let $\varphi\in C_c^\infty(B_{\bar{r}})$ be a cutoff function that is identically equal to $1$ in ${B}_{\bar{r}/2}$, and for $x \in \R^{n+1}$, let 
$$
Y(x) := \varphi(x) \eta(f(x)) X(x)
$$ 
where $X(x) := \nabla\Phi(\nabla f(x))$. By Lemma~\ref{L:fnc f}(ii), $Y \in \TVF$. By  Lemma~\ref{L:fnc f}(iv), \eqref{eqn: f ge eta}, and the definitions of $\eta$ and $\varphi$, $\spt(Y) \subset B_1$ and $\spt (Y)\cap W$ is contained in the set $\Omega_{\bar{\eta}}:={B}_{\bar{r}/2} \cap \{0 \leq \eta \circ f \leq \bar{\eta}\} $. 
 Since $f(0)=0,$ we may choose $\bar{\rho} \in (0, \bar{r}/2)$  small enough such that (again using Lemma~\ref{L:fnc f}(iv)) $0 \leq f(x)\leq \bar{\eta}/2$ for all $x \in B_{\bar{\rho}} \cap W$.  In particular, $|Y(x)| \geq m_\Phi \bar{\eta}/2$ on $B_{\bar{\rho}}\cap W$, so the first part of the claim holds with $\bar{c} \leq m_\Phi \bar{\eta}$.
\\
 
\noindent{\it Step 2: Bounding $B_\Phi :DY$ above in $\Omega_{\bar{\eta}}\times \bbS^n$.} 
Direct computation shows that $DY =\eta\circ f\, DX - X\otimes\nabla f$
for $x \in \Omega_{\bar{\eta}}$. Thus for any $(x,\nu ) \in \Omega_{\bar{\eta}} \times \bbS^n$, 
\begin{align}\label{eqn: first var in sec 33}
B_\Phi(\nu) : DY(x)  =  \eta \circ f(x)\, \underbrace{ B_\Phi(\nu):DX(x) }_{=:I} - \underbrace{ B_\Phi(\nu): X(x)\otimes\nabla f(x)}_{=:II}.
\end{align}
We will show that up to choosing $\bar{\eta}$ sufficiently small, the right-hand side of \eqref{eqn: first var in sec 33} is bounded above by $-\bar{c}|Y(x)|$. We first bound the terms $I$ and $II$ separately.
First, recall from Lemma~\ref{L:fnc f}(iii) that each level set $\{f=a\}$ is $C^{2,\alpha}$ in $B_{\bar{r}}$. So, observing that $X(x) = \nu_{\{f<a\}}^\Phi(x)$ for $x \in B_{\bar{r}}$ and $a =f(x)$,  the pointwise identity \eqref{E:MC 0-level set} applied to $\Omega = \{f<a\}$ gives
$$
B_\Phi\left(\tfrac{\nabla f(x) }{|\nabla f(x)|}\right) : DX(x) = {\Phi\left( \tfrac{\nabla f(x) }{|\nabla f(x)|}\right)} \, H_{\{f<a\}}^\Phi(x) .$$
By Lemma~\ref{L:fnc f}(iii), the right-hand side is bounded above by $-m_\Phi$ for each $x \in B_{\bar{r}}$.
Since $\Phi$ is  $C^2$ (and thus $B_\Phi\in C^1$), we have  $|B_\Phi(\nu) - B_\Phi(\sfrac{\nabla f(x)}{|\nabla f(x)|})| \leq c_{1}'\, \rmd(\nu, x)$ for a constant $c_{1}' = c_{1}'(\Phi) >0$, where we let
\[ 
    \rmd (\nu, x) := \min\left\{ \left| \nu - \tfrac{\nabla f(x)}{|\nabla f(x)|}\right|, \left| \nu + \tfrac{\nabla f(x)}{|\nabla f(x)|}\right| \right\}\,
    \]
for $\nu\in\bbS^n$ and $x \in B_{\overline{r}}.$ 
So, letting  $c_1 = c_1(\bar{r},\Phi,f) := c_{1}'\sup\{|DX_x|: x\in B_{\bar{r}}\}$,  we have
\begin{align}\label{eqn: integrand A bound}
 I =    B_\Phi(\nu):DX(x)  & \leq - m_\Phi + c_{1}\, \rmd(\nu,x).
\end{align}

Next we bound the term $II$ from below as follows:
\begin{align*}
 II &= \Phi(\nu)\Phi(\nabla f(x)) - (\nu\cdot\nabla\Phi(\nabla f(x)))(\nabla\Phi(\nu)\cdot \nabla f(x)) \geq c_{2}'\rmd^2(\nu, x)|\nabla f(x)|,
\end{align*}
where in the equality we use the zero-homogeneity of $\nabla\Phi$ and inequality follows \eqref{L: Unif Elliptic} for a constant $c_{2}' =c_{2}'(\Phi)>0$. Since $|\nabla f(x)| \geq 1/2M_\Phi$ in $B_{\bar{r}}$ by Lemma~\ref{L:fnc f}(iii), we let $c_2 = c_2( \Phi, f):= c_{2}'/2M_\Phi > 0$ to find 
\begin{equation}\label{eqn: B bound}
  II=    B_\Phi(\nu):X(x)\otimes\nabla f(x) \geq   c_2 \,\rmd^2(\nu, x).
\end{equation}
Putting \eqref{eqn: integrand A bound} and \eqref{eqn: B bound} together, 
and introducing the short-hand
notation $\ell_x := \eta(f(x))$,
we see that for $(x, \nu ) \in \Omega_{\bar{\eta}} \times \bbS^n$, the expression in \eqref{eqn: first var in sec 33} has the upper bound
\begin{align}\label{eqn: integrand}
   B_\Phi(\nu) : DY(x) \leq  - \ell_x m_\Phi + c_1 \ell_x  \, \rmd(\nu, x) - c_2 \,  \rmd(\nu, x)^2= q_x(\rmd(\nu,x)),
\end{align}
where $q_x$ is the  quadratic polynomial  defined by $q_x(t) := -m_\Phi\ell_x + c_1\ell_x\, t - c_2\,  t^2$. 
Direct computation shows that $\max_t q_x(t) = -m_\Phi \ell_x + c_1^2 \ell_x^2 /2c_2$. Since $\ell_x \leq \bar{\eta}$ for $x \in \Omega_{\bar{\eta}}$, we may choose $\bar{\eta}$ sufficiently small depending on $c_1, m_\Phi$ and $c_2$, so that $\max_t q_x(t) \leq -m_\Phi \ell_x /2$. Keeping in mind that $|Y(x)| \geq m_\Phi \ell_x$ for any $x \in \Omega_{\bar{\eta}}$, we see that \eqref{eqn: bound in claim} holds.
\end{proof}

\section{A blowup given by finitely many half-hyperplanes}\label{S: BU union of planes}
\label{S: BU analysis} 

The main goal of this section is to prove Proposition~\ref{P: flat BU under fin dens} below.
 For $\eta \in (0, \pi)$, consider the closed wedge $W_\eta$ with spine $\R^{n-1}\times\{0\}\times\{0\}$, and opening angle $\eta \in (0,\pi)$ defined by 
\begin{equation}\label{eqn: wedge 3.2}
 W_\eta:= \Big\{x\in \overline{H}\ : \  \mfrac{x_{n+1}}{\tan \eta} \leq  x_n \Big\}.
\end{equation}
Here we interpret $\frac{1}{\tan \pi/2} = \frac{\cos \pi/2}{\sin \pi/2}  = 0$. We also set $\Pi_{\eta}:= \partial W_{\eta}\cap H$.

\begin{proposition}\label{P: flat BU under fin dens}
Let $V$ be a rectifiable $n$-varifold in $\R^{n+1}$ such that 
\begin{equation}\label{E: stat condition prop BU are 1/2planes}
{\delta_\Phi V(X)}= 0\qquad \mbox{ for all }X\in\TVF.
\end{equation}
In addition, we assume that {there exists $\e>0$ such that }
\begin{align}
	\Theta^*_n(V, 0) &< +\infty, \label{A:Finite-Density}\\
    \Theta^n(V, p) & \in \mathbb{N}  \text{ for } \| V\|\llcorner H\text{-a.e. }p,\label{A:V is integral in H}\\
    \|V\|(B_r(x)) &\geq \epsilon r^n, \mbox{ for }  \|V\|\text{-a.e. } x \in B_{\epsilon}\cap H\mbox{ and }r < x_{n+1}\,.\label{A:unif lower bound mass ratio}
\end{align}
If the support of $\|V\|$ is contained in $\partial H\cup W_\eta$ for a wedge $W_\eta$ of opening angle $\eta\in (0,\pi)$, {then there exist $V_0 \in \mathcal{B}_0(V)$, $J\in\mathbb{N}\cup\{0\}$,  $\{\eta_1,\ldots, \eta_J\}\subset (0, \eta]$,  and $\{\theta_1,\ldots, \theta_J\}\subset \mathbb{N}$  satisfying}
\begin{align}
    \label{eqn: blowup union of planes}
	&V_0 \llcorner (H\times \bbS^n ) =  \sum_{i=1}^J \theta_i\,  \var(\Pi_{\eta_i}),\quad\mbox{ where }\Pi_{\eta_i} := \partial W_{\eta_i}\cap H\,,\\
    \label{eqn:blowupwetting}
  &  V_0\llcorner ( \partial H \times \bbS^n) = \|V_0\| \llcorner \partial H \otimes \Big(\frac{1}{2} \delta_{\nu_H} + \frac{1}{2} \delta_{-\nu_H}\Big).
\end{align}
\end{proposition}
While the proposition is stated for blow-ups at the origin and for wedges with a specific spine and axis, it is clear that the statement holds in a more general form after applying translations and rotations.

The proof of Proposition~\ref{P: flat BU under fin dens} follows the ideas in \cite[Lemma 5.4]{philippis2015regularity} for sets of finite perimeter that  minimize an anisotropic capillary energy. Similar ideas arose earlier in \cite[Lemma 4.5]{Hardt77} and have since been used in \cite[Lemma A.2]{DPDRminmax} and  \cite[Lemma 5.2]{DPDRYminmax}.

First, we need two lemmas. The first shows that for a stationary varifold with support contained in $\partial H$, the Grassmannian part of $V$ is ``rectifiable". This will give additional information on the piece of the blow-up {in Proposition~\ref{P: flat BU under fin dens}} that lies in $\partial H$.
\begin{lemma}\label{L: rectifiability on Gr component}
Fix $x_0 \in \partial H$ and $r>0$. If $V$ is an $n$-varifold in $\R^{n+1}$ with $\spt(\| V\|)\cap B_r(x_0) \subset \partial H$ that satisfies $\delta_\Phi V(X) = 0$
    for any $X \in \TVF$, then
    \[
    V \llcorner (B_{r/2}(x_0) \times \mathbb{S}^n ) = {\| V\|} \llcorner B_{r/2}(x_0) \otimes \Big(\frac{1}{2} \delta_{\nu_H} + \frac{1}{2}\delta_{-\nu_H} \Big).
    \]
\end{lemma}
\begin{proof}
    Let $\varphi\in C_c^\infty ( [-r,r])$ be a nonnegative cutoff function with $\varphi \equiv 1$ on $[0,r/2]$ and fix $\eta \in C_c^\infty ([-1,1])$ with $\eta(0)=0$ and $\eta'(0) =1$. Writing $x = (x' ,x_{n+1}) \in \R^n \times [0,\infty)$ for $x \in \overline H$, consider the vector field
    \[
    X(x) = \varphi(|x' - x_0'|)\, \eta(x_{n+1})\, \nabla \Phi(\nu_H).
    \]
    Note that $X \in \TVF$ since $\eta(0)=0$, and $\eta^\prime(0)=1$ ensures $DX(x)  = \varphi (|x'-x_0'|) \nabla\Phi(\nu_H) \otimes \nu_H$ for $x \in \partial H$. Therefore, using that $\spt(\|V\|)\cap B_r(x_0) \subset\partial H$, we obtain
      \begin{align*}
         0 = \delta_\Phi V(X) &= \int \varphi(|x'-x_0^\prime|)\left[ \Phi(\nu_H) \Phi(\nu) - (\nabla \Phi(\nu_H) \cdot \nu) (\nabla \Phi(\nu) \cdot \nu_H)\right] \, \rmd V(x,\nu).
    \end{align*}
    
 Using the disintegration theorem for measures, there is a family $\{\mu_x\}_{x \in \overline{H}}$ of measures on $\bbS^n$ such that
       \begin{align*}
         0& =  \int \varphi(|x'-x_0^\prime|)\,\Big( \int \underbrace{\left[ \Phi(\nu_H) \Phi(\nu) - (\nabla \Phi(\nu_H) \cdot \nu) (\nabla \Phi(\nu) \cdot \nu_H)\right]}_{=:A} \, d\mu_x(\nu) \Big)\, {\rmd\|V\|}(x).
     \end{align*}
By the Fenchel inequality, the integrand $A$ is nonnegative and equals zero if and only $\nu = \pm \nu_H$. Since $\varphi$ is nonnegative and is strictly positive on $[0,r/2]$, we find that for $\|V\|$-a.e. $x \in B(x_0,r/2)$, the measure $\mu_x$ is supported on $\{\nu_H, -\nu_H\} \subset \bbS^n$. Hence, we conclude the proof thanks to the symmetry property of varifolds.
\end{proof}

\begin{lemma}\label{L:half planes are stationary}
Let $\eta \in (0,\pi)$ and $\omega_{\eta}= - \nu_{\eta}^\Phi\cdot \nu_H$ where $\nu_{\eta}$ is the normal to $\Pi_{\eta}$ with $\nu_{\eta}\cdot e_n >0$. Then $V_\eta := \var(\Pi_{\eta}) - \frac{\omega_{\eta}}{\Phi(\nu_H)}\var(\partial H \setminus W_\eta)$ satisfies $\delta_{\bPhi} V_{\eta}(X) = 0$ for any $X\in \TVF$.
\end{lemma}
\begin{proof}
Fix $X$ any compactly supported vector field in $\overline{H}$ tangent to $\partial H$. Integrating by parts, we obtain the following equality:
\[
	\Fv V_{\eta}(X) = \int_{\Pi_{\eta}\cap H}B_\Phi(\nu_{\eta}): DX(x) \rmd \cH^n(x) - \omega_{\eta}\int_{\Pi_{\eta}\cap\partial H}X(x)\cdot e_n\rmd\cH^{n-1}.
\]
Next, denote by $\nu_0$ the outward unit normal to $\Pi_{\eta}\cap H$ in $\Pi_{\eta}$, integrating by parts, we get
\[
	\Fv V_{\eta}(X) = \int_{\Pi_{\eta}\cap\partial H}X(x)\cdot B_{\Phi}(\nu_{\eta})\nu_0\rmd \cH^{n-1}(x) - \omega_{\eta}\int_{\Pi_{\eta}\cap\partial H}X(x)\cdot e_n\rmd\cH^{n-1}.
\]
Then, we have $B_{\Phi}(\nu_{\eta})\nu_0 =\omega_{\eta}e_n$. Next, write $\nu_0 = a\nu_H +b\nu_{\eta}$ and notice that $b=(\nu_0\cdot e_n)/(\nu_{\eta}\cdot e_n)$. Therefore,
$$a = \nu_{\eta}\cdot e_n - b\nu_{\eta}\cdot\nu_H = \frac{(\nu_{\eta}\cdot e_n)^2 - (\nu_0\cdot e_n)(\nu_{\eta}\cdot \nu_H)}{\nu_{\eta}\cdot e_n} = \frac{1}{\nu_{\eta}\cdot e_n}.$$

From this, we conclude the proof as follows
$$ e_n\cdot B_{\Phi}(\nu_{\eta})\nu_0 = -a (\nu_{\eta}\cdot e_n)(\nu_{\eta}^\Phi\cdot \nu_H) = - \nu_{\eta}^\Phi\cdot \nu_H. $$
\end{proof}

We now prove Proposition~\ref{P: flat BU under fin dens}.

\begin{proof}[Proof of Proposition~\ref{P: flat BU under fin dens}]
We proceed in several steps. After some initial observations, we prove \eqref{eqn: blowup union of planes} in Steps 1-4 with an induction argument before establishing \eqref{eqn:blowupwetting} in Step 5.\\

{\it Step 0:} The collection $\mathcal{B}_0(V)$ of blowups of $V$ at $0$ is nonempty and compact by \eqref{A:Finite-Density}, and each $V_0 \in \cB_0(V)$ has  $\spt(\|V_0\|)\subset W_\eta \cup \partial H$  and $\Theta^*_n(V_0,0) \leq  \Theta^*_n(V,0)$. By \eqref{E: stat condition prop BU are 1/2planes}, any $V_0\in \cB_0(V)$ satisfies $\delta_\Phi V(X) =0$ for $X \in \TVF$.

Moreover, for any $V_0 \in \cB_0(V)$,  the restriction $V_0\llcorner (H\times \bbS^n)$ is an integral varifold, i.e. \eqref{A:V is integral in H} holds for $V_0$. To see this, let $V_0 \in \cB_0(V)$ be a blowup arising  as a limit of $V_{k}:= \left(\iota_{0, r_k}\right)^{\#} V$ (c.f. \eqref{eqn: blowup seq}) for a sequence of scales $r_k\searrow 0$. If $x\in \spt(\|V_0\|)\cap H$, there is $\spt( \|V_k\|) \ni r_k^{-1}x_k\to x$ (thus  $\spt(\|V\|) \ni x_k\to 0)$ such that $V$ satisfies the assumption \eqref{A:unif lower bound mass ratio} at each $x_k$. By \eqref{A:unif lower bound mass ratio}, for any $0<s<x \cdot e_{n+1}$, we have
\[ \frac{\|V_0\|(B_s(x))}{\omega_n s^n} = \lim_{k\to\infty}\frac{\|V_{0,r_k}\|(B_s(x))}{\omega_n s^n} = \lim_{k\to\infty}\frac{\|V\|(B_{r_k s}(r_k x))}{\omega_n (r_k s)^n} \geq \epsilon >0.\]
Therefore, $\Theta_*^n(V_0, x) > 0$ for any $x\in \spt(\|V_0\|)\cap H$. This and the integrality assumption \eqref{A:V is integral in H} for $V$ allow us to apply \cite[Theorem 4.1]{de2018minimization} which ensures that $V_0$ satisfies \eqref{A:V is integral in H}.
\\

{\it Step 1:} For each $\bar{V} \in \cB_0(V)$, let 
\begin{equation*}
	{\eta}(\bar{V}):= \inf\big\{\alpha : \spt(\|\bar{V}\|) \cap H \subset W_{\alpha} \big\} \in [0, \eta]
\end{equation*}
be the angle of the smallest wedge of the form \eqref{eqn: wedge 3.2} containing $\spt\|V_0\|\cap H$. It is not difficult to verify that $\bar{V}\mapsto {\eta}(\bar{V})$ is lower semicontinuous with respect to varifold convergence. 
So, since $\cB_0(V)$ is compact, this function achieves its minimum ${\eta}_1$ at some ${V}_1 \in\cB_0(V)$. If ${\eta}_1= 0$, then $\spt(\|{V}_1\|)\cap H = \emptyset$ and thus $\spt(\|{V}_1\|)\subset\partial H$. In this case,  we take $V_0 = V_1$ and $J=0$ and the first conclusion \eqref{eqn: blowup union of planes} holds vacuously.\\

{\it Step 2:} We henceforth assume ${\eta}_1\in (0,\eta]$.
We first claim that 
\begin{equation}\label{eqn: claim hyperplane contain}
    \Pi_{\eta_1} \subset \spt(\|{V}_1\|).
\end{equation}
To this end, consider the lower semicontinuous function $w:  \R^{n-1}\times\R_+\to (-\infty, +\infty]$ defined  by
\begin{equation*}
	w(z) = \inf \big\{ t\in \R: (z^\prime, t, z_{n+1}) \in \spt(\|{V}_1\|){\cap W_{\eta_1}} \big\} \quad\mbox{ for } z = (z^\prime, z_{n+1})\in \R^{n-1}\times\R_+.
\end{equation*}

{Notice that $w(0)\geq 0$ and l}et  $\bp_n(x) := (x_1,\ldots, x_{n-1}, x_{n+1})$ for $x\in\R^{n+1}$. By construction,  
 $\spt(\|{V}_1\|){\cap W_{\eta_1}}$ is contained in the epi-graph  $\{ (z', t, z_{n+1}) :  z = (z^\prime, z_{n+1})\in \R^{n-1}\times\R_+, t\geq w(z)\}$, {and $(z',w(z), z_{n+1}) \in \spt(\|{V}_1\|)$ whenever $w(z)<\infty$}.
In particular, the graph of $w$ is contained in the wedge $W_{{\eta}_1}$, i.e.,  
\begin{align}\label{E:w_F:defi}
   w(\bp_n(x)) \geq \frac{x_{n+1}}{\tan {\eta}_1} \quad \text{ for } x \in W_{{\eta}_1}\,.
\end{align}

Now, suppose by way of contradiction that \eqref{eqn: claim hyperplane contain} does not hold, and choose a point $\hat{x}\in \Pi_{\eta_1}$ with $\hat{x} \not \in  \spt(\|{V}_1\|)$. Since $\spt\|{V}_1\|$ is closed, strict inequality holds in \eqref{E:w_F:defi} for $x=\hat{x}.$ 
So, setting $r:= |\bp_n(\hat{x})|$,  $\tilde{D}_{r}^+ := \{x\in H : x_n = 0\mbox{ and }|x| < r\}\subset \R^{n+1}$, and $D_{r}^+:=\bp_n(\tilde{D}_{r}^+)\subset\R^{n-1}\times\R_+$, we may find a function  $\varphi \in C^{1,1}( D_{r}^+)\cap \mathrm{Lip}(\overline{ D}_{r}^+)$ satisfying
\begin{align}
	w(z) \geq \varphi(z)\geq  \frac{z_{n+1}}{\tan{\eta}_1}, \quad \forall z\in \partial D_{r}^+ \label{E:w_Fmax-phi}, \\
	\varphi (\bp_n(\hat{x})) > \frac{\hat{x}_{n+1}}{\tan\eta_1}
  \label{E:phi-in-between},\\
	\varphi \equiv 0 \mbox{ on } D_{r}\cap\{ (z', z_{n+1}) \in \R^{n-1} \times \R_+:z_{n+1}  = 0\}. \label{E:phi-vanishes-on-bdrH}
\end{align}

Let $\Phi^\# : \R^n \to \R$ be the one-homogeneous function defined on $\bbS^{n-1}$ by $\Phi^{\#} (\nu) := \Phi(\nu, -1)$.
By part two of \cite[Lemma 2.11]{philippis2015regularity} applied to $e=e_{n+1}, H, \varphi$,  there is a solution $u \in$ $C^{1,1}\left(D_{r}^{+}\right) \cap \mathrm{Lip}(\overline{D}_{r}^+)$ to
\begin{equation}\label{E:PDE}
	\begin{cases}\operatorname{div}(\nabla \Phi^\# (\nabla u))=0, & \text { in } D_{r}^{+}, 
	\\ u=\varphi, & \text { on } \partial D_{r}^{+}.\end{cases}
\end{equation}

We claim that $u \leq w$ in $\overline{D}_{r}^+$ by the maximum principle. Indeed, suppose by way of contradiction that the lower semicontinuous function $f:= w - u$ on $\overline{D}_{r}^+$ has $f(\hat{z})<0$ at a point $\hat{z} \in \overline{D}_{r}^+$ where $f$ achieves its minimum.  Note that $\hat{z} \in D_{r}^+$ by \eqref{E:w_Fmax-phi}.
Let $\hat{u} := u - |f(\hat{z})|$. Since $\hat{u}$ solves \eqref{E:PDE} (with boundary data $\varphi - |f(\hat{z})|$), its graph induces a varifold $V_{\hat{u}}$ in the open cylinder $C_*:=\{ x \in H : \bp_n(x) \in D_{r}^+\}$ satisfying $\delta_\Phi V_{\hat{u}}(X)=0$ for all $X \in C^1_c(C_*, \R^n).$
Moreover, $(\hat{z}^\prime, \hat{u}(\hat{z}), \hat{z}_{n+1})\in \spt(\|V_{\hat{u}}\|) \cap \spt(\|{V}_1\|)$ and $\spt(\|{V_1}\|)$ lies in one connected component of $C_*\setminus \spt(\|V_{\hat{u}}\|)$. Thus $\spt(\|V_{\hat{u}}\|) \subset \spt(\|{V}_1\|)$ by Solomon-White's maximum principle  \cite{solomon1989strong}.
Since $\spt(\|{V}_1\|){\cap W_{\eta_1}}$ is contained in the epi-graph of $w$, this guarantees that $\hat{u} \geq w$ on $\overline{D}^{{+}}_{r}$, a contradiction to the assumptions that $\hat{u}(0) = - |f(\hat{z})|<0$ and $w(0)\geq 0$. Thus $u \leq w$.

Now, let $\widetilde{V}\in\cB_0(V_1)$. 
Since  $\spt(\|{V}_1\|) \cap C_*$ is contained in the epi-graph of $u$ (as $u \leq w$ on $ \overline{D}_{r}^+$) and $u\equiv 0$ on {$\{z_{n+1}=0\}$} by \eqref{E:phi-vanishes-on-bdrH}, we find that
	\begin{equation*}\label{E:noncrossing-containment2}
			\spt(\|\widetilde{V}\|) \subset \left\{ x : x_n \geq \partial_{z_{n+1}}u(0) x_{n+1} \right\}\,.
	\end{equation*}	
	Applying Hopf lemma to $u$ and using \eqref{E:w_Fmax-phi}, we find that $\partial_{z_{n+1}}u(0) > 1/\tan{\eta_1}$, i.e. $\spt(\|\widetilde{V}\|)\subset W_{\tilde{\eta}}$
    where  $\tilde{\eta} < {\eta}_1 \in (0,\pi)$ has $\tan \tilde{\eta} = 1/\partial_{z_{n+1}}u(0)$. 
    This contradicts the definition of ${\eta}_1$. We conclude that \eqref{eqn: claim hyperplane contain} holds.

   Now, since ${V}_1\llcorner (H\times \bbS^n)$ is an integer rectifiable varifold and \eqref{eqn: claim hyperplane contain} holds, 
   \[
   \hat{V}_1 := (V_1 - \var(\Pi_{{\eta}_1}) )\llcorner(H\times \bbS^n)
   \]
   is a (non-negative) integer rectifiable varifold, and $\spt(\|\hat{V}_1\|) \subset W_{\eta_1}$. Moreover, $\delta_\Phi \hat{V}_1 (X) =0$ for every $X \in C^1_c(H; \R^n)$,  since $\delta_\Phi{V}_1(X) = \delta_\Phi \var(\Pi_{{\eta}_1})(X) = 0$ for all such vector fields.  (Note that this does not hold for all $X \in \TVF$, as $\delta_\Phi \var(\Pi_{{\eta}_1})$ does not vanish along all $X \in \TVF$ unless $\eta_1= \pi/2$.)\\
    
{\it Step 3:} We now argue by induction.  For $N\in\mathbb{N}$ with $N\geq 2$, assume there are $V_{N-1} \in \cB_0(V)$ and ${\eta}_1 \geq \dots \geq {\eta}_{N-1}>0$ such that 
\begin{equation}
    \label{eqn: IH part 1}
\hat{V}_{N-1} := \Big[V_{N-1} - \sum_{i=1}^{N-1} \var(\Pi_{{\eta}_i})\Big] \llcorner (H \times \bbS^n)
\end{equation}
is a (non-negative) integer rectifiable varifold satisfying 
\begin{equation}
    \label{eqn: IH part 2}
\spt(\|\hat{V}_{N-1}\|)\subset W_{{\eta}_{N-1}} \qquad \text{ and } \qquad \delta_\Phi \hat{V}_{N-1}(X) = 0 \text{ for all }X \in C^1_c(H; \R^n).
\end{equation}
By Steps 1 and 2, this inductive hypothesis holds for $N=2$ (in the case $\eta_1>0$ as above).

Analogously to Step 1, let  ${V}_N \in \cB_0(V_{N-1})\subset \cB_0(V)$ be a minimizer of the lower-semicontinuous function 
\[
{\eta}(\bar{V}) =\inf\Big\{ \alpha :  \spt\Big(\Big\| \Big[\bar{V} - \sum_{i=1}^{N-1} \var(\Pi_{{\eta}_i})\Big]\Big\|\llcorner H\Big) \subset W_{\alpha} \Big\} \in [0, \eta_{N-1}]
\]
in the compact set $\cB_0({V}_{N-1})$ and set  ${\eta}_N= {\eta}({V}_N)$. Let  $\tilde{V}_N= [V_{N} - \sum_{i=1}^{N-1} \var(\Pi_{{\eta}_i})] \llcorner (H \times \bbS^n)$ and observe that $\tilde{V}_N$ is a (non-negative) integer rectifiable varifold that satisfies \eqref{eqn: IH part 2} with $\tilde{V}_N$ in place of $\hat{V}_{N-1}$.

If ${\eta}_N = 0$, then $V_N\llcorner (H\times \bbS^n) =  \sum_{i=1}^{N-1} \var(\Pi_{{\eta}_i})$. In this case, we take $V_0= V_N$ and see that, up to relabeling so that multiplicities in $\eta_i$ are expressed through the integer-valued multiplicities $\theta_i$, the first conclusion \eqref{eqn: blowup union of planes} holds.

Next suppose $\eta_N>0.$ By repeating the argument of Step 2---which only used that $\delta_\Phi V_1(X) =0$ for $X \in C^1_c(H; \R^n)$---with $\tilde{V}_{N}$ in place of $V_1$, we find that $\Pi_{\eta_N} \subset \spt(\|\tilde{V}_{N}\|).$  Together with the fact that $\tilde{V}_N$ is integer rectifiable, this guarantees that 
\[
\hat{V}_N : = \tilde{V}_N-  \var(\Pi_{{\eta}_N}) = \Big[V_{N} - \sum_{i=1}^{N} \var(\Pi_{{\eta}_i})\Big] \llcorner (H \times \bbS^n)
\]
is a (non-negative) integer rectifiable varifold with $\spt(\|\hat{V}_N\|) \subset W_{\eta_N}$. Moreover, for all $X \in C^1_c(H,\R^n)$, we have $\delta_\Phi \tilde{V}_N (X)= \delta_\Phi \var(\Pi_{{\eta}_N})(X)=0$ and the same holds for $\delta_\Phi \hat{V}_N.$ In other words, \eqref{eqn: IH part 1}-\eqref{eqn: IH part 2} hold with $N$ in place of $N-1.$\\

{\it Step 4:} We now show that induction terminates after finitely many iterations, i.e. ${\eta}_N= 0$ for some $N \in \mathbb{N}$. Note that if $\eta_N>0,$ $V_N{\llcorner (H \times \bbS^n)}$ is the sum of the varifolds $\sum_{i=1}^N \var(\Pi_{\eta_i})$ and $\hat{V}_N$ as defined in \eqref{eqn: IH part 1}. Thus  $\Theta_n^*({V}_N, 0 ) \geq \Theta_n(\sum_{i=1}^N \var(\Pi_{\eta_i}), 0 ) = N/2$, where the second identity holds since that the varifold induced by $\Pi_{\eta}$ has density $1/2$ for any $\eta \in (0,\pi)$. On the other hand, as observed in Step 0, $\Theta_n^*({V}_N, 0) \leq \Theta_n({V}, 0)=:\bar\Theta$, where $\bar{\Theta}<\infty$ by assumption \eqref{A:Finite-Density}. Thus, $\eta_N=0$ for $N > 2\bar{\Theta}.$ Thus, by Steps 1-4, we have obtained $V_0\in \cB_0(V)$ (specifically, $V_0=V_N$ for the first $N$ with $\eta_N=0$) that satisfies \eqref{eqn: blowup union of planes}.\\

{\it Step 5:} We now prove that $V_0$ satisfies \eqref{eqn:blowupwetting} using Lemma~\ref{L: rectifiability on Gr component} and the structure \eqref{eqn: blowup union of planes} of $V_0 \llcorner(H \times \bbS^n)$. For $i=1,\dots, N,$ consider the (possibly signed)  varifolds 
$$V_{\eta_i} := \var(\Pi_{\eta}) - \frac{\omega_{\eta_i}}{\Phi(\nu_H)}\var(\partial H \cap \{x_n <0\})$$
with $\omega_{\eta_i} = -\nu_H \cdot \nabla \Phi((0,\dots,0, \sin \eta_i, -\cos \eta_i))$. Recall from Lemma~\ref{L:half planes are stationary} that $\delta_{\bPhi} V_{\eta_i}(X) = 0$ for any $X\in \TVF$.
Consequently, $V' = V_0 - \sum_{i=1}^N {V}_{\eta_i}$ is a signed varifold with $\spt (\|V'\|) \subset \partial H$ satisfying $\delta_\Phi V' (X) = 0$ for all $X \in \TVF.$ To turn this 
into a ({non-negative}) varifold to which we can apply Lemma~\ref{L: rectifiability on Gr component}, let $I \subset \{ 1,\dots ,N\}$ be the set of indices such that $\omega_{\eta_i}<0$, and let $V'' = V' - \sum_{i\in I} \frac{\omega_{\eta_i}}{\Phi(
\nu_H)} \var(\partial H )$. Now $V''$ is a (non-negative) varifold with $\spt(\|V''\|)\subset \partial H$ and $ \delta_\Phi V''(X) =0$ for all $X \in \TVF$. So, Lemma~\ref{L: rectifiability on Gr component} guarantees that 
$V'' = \|V''\| \otimes (\frac{1}{2} \delta_{\nu_H} + \frac{1}{2} \delta_{-\nu_H}).$ By construction, the same property holds for $V_0$, i.e. \eqref{eqn:blowupwetting} holds.
\end{proof}

\begin{remark}\label{rmk: mono formula}
In the case of the area functional,  assumptions \eqref{A:Finite-Density} and \eqref{A:unif lower bound mass ratio} are straightforward consequences of the classical monotonicity formula for stationary varifolds. In fact, in Step 0 and in the induction Step 3 we used  \eqref{A:Finite-Density}  to show $\cB_0(V)$ is non-empty, and we used assumption \eqref{A:unif lower bound mass ratio} together with the integrability theorem \cite[Theorem 4.1]{de2018minimization} to show 
 that every $V_0 \in \cB_p(V)$ has integer density $\|V_0\|$-a.e. on $H$; these are classical properties for the isotropic setting that we show in  Lemma~\ref{lemma: cone}.
 Moreover, one could give a substantially simplified proof of Proposition~\ref{P: flat BU under fin dens} in this case by tilting hyperplanes and using the fact that any blowup is a cone to guarantee that the varifold cannot touch the hyperplane solely at infinity.
\end{remark}

\begin{lemma}\label{lemma: cone}
    Let $V$ be a rectifiable $n$-varifold with $\spt\|V\| \subset \overline{H}$ such that, for some $\lambda \in\R$,
    \begin{equation}\label{eqn: isotropic stationary for the weighted varifold}
    \int  \left(\mathrm{Id} - \nu\otimes\nu\right) : DX_x \, \rmd V(x,\nu)  =  \lambda \int  X(x) \cdot \nu \, \rmd V(x,\nu)  \qquad \text{ for all }X \in \TVF\,.
        \end{equation}
For any $p \in \spt \| V\|\cap \partial H$, the blowup set $\cB_{p}(V)$ is nonempty and any $V_0 \in \cB_{p}(V)$ satisfies \eqref{eqn: isotropic stationary for the weighted varifold} with $\lambda =0$. Furthermore its support $\spt\|V_0\|$ is a cone in $\overline{H}$.
Moreover, if $\Theta^n(V,p) \in \mathbb{N}$ for $\| V\| \llcorner H$-a.e. $p$, then the same is true for any $V_0 \in \cB_p(V)$. 
\end{lemma}

\begin{proof} The proof of the classical (almost) monotonicity formula for varifolds with bounded mean curvature, see \cite[Theorem 5.1]{allard1972first}, is based on testing the first variation with a vector field $$X(x) = \varphi(|x-p|)(x-p)$$ for a point $p \in \spt \|V\|$ and a smooth compactly supported cutoff function $\varphi$. For a varifold $V$ as in the statement of the lemma and for $p \in \spt \|V\| \cap \partial H$, the restriction of this vector field to $\overline{H}$ lies in $\TVF$ and thus is admissible for the stationarity condition \eqref{eqn: isotropic stationary for the weighted varifold}. Thus, repeating the classical proof of Allard shows that the density ratio $\frac{\| V\|(B_r(p))}{\om_nr^n}$ is almost-monotone. This is also observed in \cite[Lemma 2.9, Remark 2.11]{nick2024regularity} and in \cite[Section 3]{wang2024allard}. Classical arguments then show that any blowup sequence for $V$ converges subsequentially to a varifold $V_0 \in \cB_0(V)$ satisfying 
\eqref{eqn: isotropic stationary for the weighted varifold} with $\lambda =0$  with $\spt \|V_0\| \subset \overline{H}$ such that the density $\frac{\| V_0\|(B_r)}{\om_nr^n}$ at the origin is constant in $r>0$. Since this density ratio is constant, further classical arguments show $\spt \| V\|$ is a cone. The final statement follows from the closure theorem for integral varifolds with bounded first variation see \cite[Theorem 6.4]{allard1972first} or \cite[Section 42.8]{simon2014introduction}.
\end{proof}

\section{Stationary sets are viscosity subsolutions {of Young's law}}\label{sec: stat to visc}
{This section is dedicated to the proof of Theorem~\ref{thm: stationary implies viscosity}.}
The proof will be obtained by a contradiction argument taking the blow-up given by Proposition~\ref{P: flat BU under fin dens} and constructing a suitable tangential vector field that increases energy.

\begin{proof}[Proof of Theorem~\ref{thm: stationary implies viscosity}]

\noindent{\it Step 1: Set up.} {Let $E$ be a set as in the theorem statement satisfying \eqref{eqn: aniso isotropic stationary}.}
Suppose by way of contradiction there is a $C^1$ domain $G \subset \R^{n+1}$ touching $E$ from the interior at $p \in \Gamma$ such that 
	\begin{align}\label{eqn: contra hp}
\nu_{G}^\Phi(p) \cdot \nu_H> -\omega \geq 0,
 \end{align}
and that $p\notin T_E$. Since $G$ is of class $C^1$, the blowup of the open set $G$ at $p$ is a half space, which we denote by $H_G$, with outer unit normal $\nu_G(p)$.   Let $W_G = \overline{H} \setminus H_G$.

By \eqref{eqn: contra hp}, we can apply Theorem~\ref{T: finiteness density} to find that $\Theta^*_n(\partial^* E, p) < +\infty$, and equivalently, $\bar{\Theta}:=\Theta^*_n(V, p)< +\infty$ for the {capillary} varifold
$$V := \var(\partial^*E \cap H)  - \frac{\omega}{\Phi(\nu_H)} \var(\partial^*E \cap \partial H)
$$
{Thus $\cB_p(V)$ is nonempty.  For any $\bar{V} \in \cB_p(V)$,  $\Theta^*_n(\bar{V}, p) \leq \bar{\Theta}$ and in particular $\bar{V}$ satisfies \eqref{A:Finite-Density}. Additionally, $\bar{V}$ satisfies \eqref{E: stat condition prop BU are 1/2planes} by scaling and the fact that $V$ satisfies \eqref{eqn:fv}. Repeating the argument of Step 0 of the proof of Proposition~\ref{P: flat BU under fin dens} and using the assumption \eqref{eqn: lower density assumption intro 1} shows that $\bar{V}$ additionally satisfies \eqref{A:V is integral in H} and \eqref{A:unif lower bound mass ratio}. Moreover, since $G\cap H\subset E$, the support of $\bar{V}$ is contained in $\partial H \cup W_G$. Up to a rotation (also replacing $\Phi$ with its rotation), we may assume  $W_G = W_\eta$, in the notation \eqref{eqn: wedge 3.2},  for $\eta:=\arccos(\nu_{G}(p) \cdot \nu_H ). $
Thus, we may apply Proposition~\ref{P: flat BU under fin dens} to an any $\bar{V} \in \cB_p(V)$ to obtain $V_0 \in \cB_0(\bar{V}) \subset \cB_p(V)$, $J\in\mathbb{N}\cup\{0\}$, $\{\eta_1,\dots, \eta_J\}\subset (0, \eta]$, and $\{\theta_1,\dots, \theta_J\}$ such that   \eqref{eqn: blowup union of planes} and \eqref{eqn:blowupwetting} hold.
} 

We {claim} that $J\neq 0$. Indeed, if $J= 0$, then in particular $\spt(\|V_0\|)\subset\partial H$. Consider the blow-up $F$ (in $L^1_{loc}$) of $E$ at {$p$} associated to the same blow-up sequence $\{r_k\}_k$ generating $V_0$. We claim that $F = H$, which then contradicts  ${p}\notin T_E$ and shows $J\neq 0$. To show that $F=H$, it suffices to prove $\partial^* F \subset \partial H$, since this implies $F$ is either empty or $H$, and $F$ cannot be empty since it contains $H_G$. By construction, $\partial^* F\setminus W\subset \partial H$. Assume by way of contradiction that there exists $x\in \partial^* F\cap W\cap H$. {Then $P(F;B_r(x)) >0$ for any $r>0$.} Since $x\in H$, we have $x\notin \spt(\|V_0\|)$ and thus $\|V_0\|(B_{r_1}(x_0)) = 0$ for some $r_1>0$. This and lower semicontinuity of the perimeter imply that $0 = \|V_0\|(B_{{r_1}}(x)) = \lim_k P(E_{0,r_k}, B_{{r_1}}(x)) \geq P(F,B_{{r_1}}(x)) > 0$. This contradiction ensures that $\partial^* F\cap W\subset \partial H$ and then $\partial^* F\subset \partial H$.

For each $i=1,\dots J$, let $\nu_i\neq \pm \nu_H$ denote the normal to $\Pi_{i}$ such that $\nu_i$ is the outer unit normal to the half space with boundary $\Pi_i$ that contains $H_G\cap H$, i.e.  $\nu_i \cdot e_n >0$. From the structure \eqref{eqn: blowup union of planes} and \eqref{eqn:blowupwetting}, 
\begin{equation}\label{E:first var V by half planes}
    0 =\delta_\Phi V_0(X)= \sum_{i=1}^J \theta_i\int_{\Pi_{i}
     \cap H}B_\Phi(\nu_i) : DX(x) \rmd \mathcal{H}^n(x) +\int_{\partial H}B_\Phi(\nu_H) : DX(x)\rmd \|V_0\|(x),
\end{equation}
for all  $X \in \TVF$. Since each $\Pi_{i}$ is contained in $W_G$, we have $\nu_{G}(p) \cdot \nu_H  \leq \nu_i \cdot \nu_H$. {Similarly, as $\nu_i \in \text{span}\{(e_n)_+, \nu_H\}$ 
by \eqref{I:contrapositive ordering Cahn-Hoffman}, we have}
$\nu^\Phi_{ G}({p}) \cdot \nu_H \leq  \nu^\Phi_i \cdot  \nu_H$. 
 This, together with \eqref{eqn: contra hp} and the assumption $\omega \leq 0$ guarantee that
\begin{equation}\label{E:contradiction-omega}
	0\geq \omega > \omega_i\qquad \mbox{  for any  }i\in \{1,\ldots, J\},
\end{equation}
where we let $\nu_i^\Phi = \nabla \Phi(\nu_i)$ and 
\begin{equation}
    \label{eqn: omega i}
    \omega_i :=- \nu^\Phi_i \cdot  \nu_H.
\end{equation}

\noindent{\it Step 2: The test vector field.} Now, let $\varphi\in C_c^{\infty}([0,+\infty))$ be a smooth cutoff function and consider the vector field $X_0 \in \TVF$ defined by
\[
X_0(x):= -\varphi(|x|)\nustar 
\]
Letting $\hat{x} := x/|x|$, direct computation shows $D X_0(x) = - \varphi^\prime(|x|)\nustar\otimes \hat{x}$. Thus, for each $i =1,\dots, J$, we have $B_\Phi(\nu_i): DX_0(x)  =  - \varphi^\prime(|x|) \hat{x} \cdot \tau_i$, where we let
\begin{equation*}
	\tau_i := \Phi(\nu_i)\nustar -  (\nustar\cdot \nu_i) \nabla\Phi(\nu_i).
\end{equation*}

{Notice that we also have $B_\Phi(\nu_H): DX_0(x)  =  - \varphi^\prime(|x|)\Phi(\nu_H) \hat{x} \cdot \nustar$ since $\nu_H\cdot\nustar = 0$.} Thus, taking $X_0$ as a test vector field in the first variation in \eqref{E:first var V by half planes}, we have 
\begin{equation*}
\begin{aligned}
    0 &= \sum_{i=1}^J \theta_i\int_{\Pi_{i} \cap H}(-\varphi^\prime(|x|))\hat{x}\cdot\tau_i \,\rmd \mathcal{H}^n(x) + \int_{\partial H}(-\varphi^\prime(|x|))\hat{x}\cdot\nustar\, \rmd \|V_0\|(x) \\
    &{\geq  \sum_{i=1}^J \theta_i\int_{\Pi_{i} \cap H}(-\varphi^\prime(|x|))\hat{x}\cdot \tau_i \,\rmd \mathcal{H}^{n}(x) - \frac{\omega}{\Phi(
    \nu_H)} \int_{\partial H\cap H_G}(-\varphi^\prime(|x|))\hat{x}\cdot \nustar\,\rmd \mathcal{H}^{n-1}}
\end{aligned}
\end{equation*}
where the inequality follows since (i) $-\varphi^\prime \geq 0$ and $\nustar\cdot \hat{x}\geq 0$ for any $x\in \partial H \setminus H_G$; and (ii) $\|V_0\|\res (H_G \cap \partial H) = -\frac{\omega}{\Phi(\nu_H)}\var(H_G \cap \partial H)$. Next, by an application of the coarea formula and setting $c_\varphi= -\int_0^1 r^{n-1}\varphi^\prime(r)\rmd r >0$, we obtain from the last displayed inequality that 
\begin{equation}\label{E:first-variation-Ani2}
\begin{aligned}
    0 \geq c_\varphi\Big( \sum_{i=1}^J \theta_i\int_{\Pi_{i} \cap H \cap \bbS^n} x\cdot \tau_i\rmd \mathcal{H}^{n-1}(x) - \frac{\omega}{\Phi(
    \nu_H)} \int_{\partial H \cap H_G\cap \bbS^n}\hat{x}\cdot \nustar\rmd \mathcal{H}^{n-1}\Big).
\end{aligned}
\end{equation}

Recall that Lemma~\ref{L:half planes are stationary} ensures that $V_i := \var(\Pi_{i}) - \frac{\omega_i}{\Phi(\nu_H)}\var(\partial H\cap H_G)$ satisfies $\delta_{\bPhi}V_i(X) =0$ for any $X\in\TVF$ where $\omega_i$ is defined in \eqref{eqn: omega i}. 

{So, evaluating $0=\delta_\Phi V_i(X_0)$, analogously applying the coarea formula, and multiplying through by $c_{\varphi}>0$,} we conclude that for every $i=1,\dots, J$,
\begin{equation}\label{E:Stationarity-Plane}
	\int_{\Pi_{i} \cap H\cap\bbS^n} \hat{x}\cdot\tau_i  \rmd \cH^n = \frac{\omega_i}{\Phi(
    \nu_H)} \int_{H_G\cap\partial H\cap\bbS^n} \hat{x}\cdot\nustar\rmd \cH^{n}.
\end{equation}
{
Applying \eqref{E:Stationarity-Plane} to  \eqref{E:first-variation-Ani2} and multiplying  through by $c_\varphi$ shows that 
\begin{equation}\label{eqn: contra}
0 \geq   \frac{\sum_{j=1}^N \theta_i \omega_i - \omega }{\Phi(\nu_H)}\int_{H_G \cap\partial H\cap\bbS^n}\hat{x}\cdot \nustar \rmd \cH^{n}.
\end{equation}
Since $ \hat{x}\cdot\nustar \leq 0$ for any $\hat{x}\in H_G \cap\partial H\cap\bbS^n$, the integral in \eqref{eqn: contra} is strictly negative, and thus $\sum \theta_i \omega_i -\omega \geq0$. Recalling that $\theta_i \in \mathbb{N},$ this contradicts \eqref{E:contradiction-omega}.
}
\end{proof}

\subsection{The hydrophilic regime}\label{R:small-contact-angle-regime}

We now explain why our techniques for the hydrophobic case cannot be generalized to treat the hydrophilic case, i.e., $\omega \in (0,1)$.
Fix $\theta_0 <\theta_1 \leq  \theta_2 \leq  \dots \leq \theta_N$.
The proof above (in the case of the area functional for simplicity) was based on the following elementary rigidity fact: if $\theta_0\geq \pi/2$ then  $\sum_{i}^N \omega_i > \omega_0$ where $\omega_i := \cos\theta_i$. This rigidity fails if $\theta_0< \pi/2$: for instance, if $N=3$ and $\theta_0<\theta_1 <\pi/2$, then we can choose $\theta_2, \theta_2$ such that $\sum_{i=1}^3 \omega_i = \omega_0$.

In this case consider $F = (H\setminus W_{\pi - \theta_1})\cup (W_{\pi-\theta_2} \setminus W_{\pi-\theta_3})$ and $V_F := \var(\partial^* F\cap H) - \omega_0 \var (\partial H\setminus W_{\pi-\theta_1})$. It is straightforward to see (cf. Lemma~\ref{L:half planes are stationary}) that $\delta V_F(X) = (\sum_{i=1}^3\omega_i - \omega_0)\delta \var (\partial H\setminus W_{\pi-\theta_1})(X) = 0$ for any $X\in \TVF$. 

 \begin{wrapfigure}{r}{0.5\textwidth}
   \begin{center}
     \includegraphics[width=0.4\textwidth]{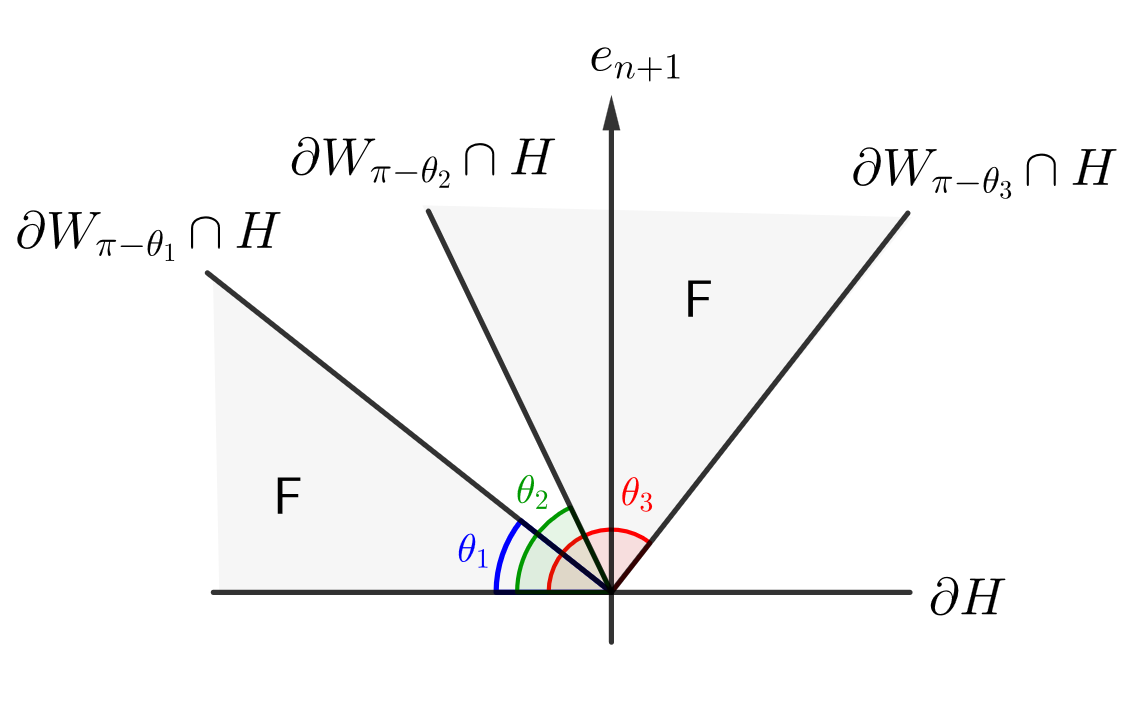}
   \end{center}  \caption{\label{fig:figure1}}
 \end{wrapfigure}
 
The key point now is that $F$ is not a viscosity sub-solution to Young's law (cf. Theorem~\ref{thm: stationary implies viscosity}), as it is clear that there exists $G$ that touches $F$ from the interior at $0$ with $\nu_G(0)\cdot \nu_H = \cos\left(\pi - \eta\right)$ for some $\eta \in (\theta_0, \theta_1)$. This shows that the argument using blowups to deduce viscosity bounds from stationarity in Theorem~\ref{thm: stationary implies viscosity} does not work in this regime.

\begin{remark}
We note that similar problems arise in recent works on the regularity of varifolds stationary for capillary energies \cite{nick2024regularity,wang2024allard}. In both cases, the authors impose assumptions that rule out examples of the type described above.
Similarly, if one adds additional assumptions to Theorem~\ref{T:Alexandrov:Isotropic}  that exclude this behavior, it would be possible to extend the proof to include the hydrophilic regime.
\end{remark}

\begin{remark}
 We observe that, up to minor modifications, the same proof of Theorem \ref{thm: stationary implies viscosity} allows to prove that, in the regime $\omega \in (0,\Phi(\nu_H))$, a set of finite perimeter satisfying \eqref{eqn: Lower Density} for some $\epsilon>0$, and that is $\bPhi_\omega$-stationary under volume-preserving variations, is a viscosity
    ``super"-solution {to Young's law}. By ``super"-solution, we mean that the condition of touching from the interior would be replaced by touching from the exterior with a set $G\supset E$ with the opposite inequality $\nu_G^\Phi \cdot \nu_H \geq -\omega$.
\end{remark}

\section{The Heintze-Karcher inequality}\label{S: HK ineq}

The main goal of this section is to prove the Heintze-Karcher inequality in Theorem~\ref{thm: HK} below. First, we need some definitions. 
Fix $\omega \in (-\Phi(\nu_H) , \Phi(\nu_H))$ and let $A \subset \overline{H}$ be a closed set.
The (anisotropic) shifted distance of a point $x \in \overline{H}$ to $A$ is 
\begin{equation*}
	\sdist^A(y) := \sup\big\{ r \geq 0: \cW_r(y+r\omega \aniEn)\cap A = \emptyset \big\},
\end{equation*}
where $\aniEn = \frac{\nabla\Phi(\nu_H)}{\Phi(\nu_H)}$.  Observe that $\sdist^A(y) = \rmd(y + \sdist^A(y)\omega\aniEn , A)$ where $\rmd(y, A) := \inf\left\{\Phi^\ast(y-x): x\in A\right \}$.
We define the {shifted (anisotropic) normal bundle} as
    $$
    N_\omega(A) \mid L:= \{(a,u)\in L\times (\partial \cW_1 -\omega \aniEn):\sdist^A(a+su)=s \mbox{ for some } s>0\},
    $$
for any $L\subset \overline{H}$. We let $N_\omega(A) := N_\omega(A)\mid A$.

We let $\reg(\rel E)$ be the set of those points $x$ in $\rel E$ such that, for some $r>0$, $\rel E \cap B_r(x)$  can be written as the graph of a $C^{2}$ function, and let $\sing(\rel E)$ be the complement of $\reg(\rel E)$ in $\rel E$.
\begin{theorem}\label{thm: HK}
Fix $\omega\in (-\Phi(\nu_H), \Phi(\nu_H))$ and let $E\subset H$ be a set of finite perimeter with finite volume. Assume that $E$ is a viscosity sub-solution for Young's law,
\begin{align}
    \cH^n\left(N_\omega(\overline{H\setminus E})\mid \left(\sing(\rel E)\cap H\right)\right) = 0, \label{A: HK normal bundles meas zero}\\
    H_E^\Phi > 0  \mbox{ on }\reg(\rel E), \mbox{ and }
	\cH^n\left(T_E\right) = 0.\label{A:HK}
\end{align}
Then 
\begin{equation}\label{eqn: HK}
	|E| \leq \frac{n}{n+1} \int_{\reg(\rel E)} \frac{\Phi(\nu_E(x)) +\omega \nu_{E}(x)\cdot \aniEn}{H_E^\Phi(x)} \rmd  \cH^n(x).
\end{equation}
\end{theorem}

\begin{remark}
Property \eqref{A: HK normal bundles meas zero} is a natural qualitative regularity assumption. Indeed in Section \ref{S:proof of Alexandrov} we first introduce $(n,h)$-sets (sets with bounded mean curvature in a viscosity sense) and we show that if $\partial E\cap H$ is a $(n,h)$-set for some $h \geq 0$ and $\cH^n(\sing(\rel E)\cap H) = 0$, then \eqref{A: HK normal bundles meas zero} is satisfied.  
\end{remark}

\subsection{Anisotropic shifted normal bundle} Before proving Theorem~\ref{thm: HK}, we establish some properties of the shifted distance and shifted normal bundle. 
As above, let $\omega \in (-\Phi(\nu_H), \Phi(\nu_H))$ be fixed and let $A\subset \overline{H}$ be a closed set. 

\begin{definition}\label{D:Lusin property} We say that $N_\omega(A)$ satisfies the \emph{Lusin condition in $H$} if and only if the following holds:
$$
 \cH^n\left(N_\omega(A) \mid L \right)=0 \mbox{ for any }L \subseteq A \cap H \mbox{ with } \cH^n(L)=0.
$$    
\end{definition}

\begin{lemma}\label{L: Lusin cond - shifted vs classical}
The set
    $$
    N_\omega(A)=\left[\mathrm{id} \times (\mathrm{id} - \omega \aniEn)\right](N_0(A))
    $$
    is $n$-rectifiable. Moreover $N_\omega(A)$ satisfies the Lusin condition in $H$ if and only if $N_0(A)$ satisfies the Lusin condition in $H$.
\end{lemma}
\begin{proof}
The inclusion $(a, u)\in N_\omega(A)$ is equivalent to say that there exist $v\in \partial \cW_1$ and $s>0$ such that $u=v-\omega \aniEn$ and $\sdist^A(a+su)=s$.
As observed above, we have $\sdist^A(a+su) = \rmd(a+su+ s\omega\aniEn , A)=\rmd(a+sv,A)$. Hence, we deduce $\rmd_0^A(a+sv,A)=s$ which is equivalent to say that
$(a,v)\in N_0(A)$. This concludes the proof of 
$$N_\omega(A)=\left[\mathrm{id} \times (\mathrm{id} - \omega \aniEn)\right](N_0(A))$$
since $u = (\mathrm{id} - \omega \aniEn) (v)$. Since $N_0(A)$ is $n$-rectifiable, see \cite[Lemma 5.2]{de2020uniqueness}, and $\mathrm{id} \times (\mathrm{id} - \omega \aniEn)$ is smooth, we deduce also that $N_\omega(A)$ is $n$-rectifiable.
The moreover part is a consequence of the equality in the first part of the lemma, using that the map $\mathrm{id} \times (\mathrm{id} - \omega \aniEn)$ is smooth, hence Lipschitz.
\end{proof}

Let  $\bU_\omega^A$ be the set on which we have a unique nearest point projection (with respect to $\sdist^A$) onto $A$, i.e.
\begin{equation*}
	\bU_\omega^A := \left \{ x\in\R^{n+1}: \exists ! x_0 =: \bp_\omega^A(x) \in A\text{ such that }\sdist^A(x) = \rmd(x + \sdist^A(x)\omega\bar{\nu}_H^\Phi,x_0)\right\}.
\end{equation*}
We call the map $\bp_\omega^A$ the nearest point projection onto $A$.

 {It is not difficult to check that $d_\omega^A$ is Lipschitz, and thus differentiable almost everywhere.} The following lemma collects some further properties of $\sdist^A.$

\begin{lemma}\label{L:AniShifDistBasicProps}
If $y \in \R^{n+1}\setminus A$ and $x\in A$ such that $\sdist^A(y) = \Phi^\ast(y + \sdist^A(y)\omega\bar{\nu}_H^\Phi - x)$, we have:
\begin{enumerate}[\upshape (i)]
	\item\label{I:AniShifDist:homogeneity} $ \sdist^A(y+\tau(x-y)) = (1-\tau)\sdist^A(y)$ for any choice of $\tau\in (0,1]$,
    
    \item\label{I:linear distance} for any $t>0$ and $u\in\partial\cW_1 -\omega\aniEn$, if $\sdist^A(x+tu) = t$, then $x+su\in\bU_\omega^A$ and $\bp_\omega^A(x+su) = x$ for any $s\in (0,t)$.
\end{enumerate}
In addition, if $y$ is a differentiability point of $\sdist^A$, we obtain:
\begin{enumerate}[\upshape (i)]\setcounter{enumi}{2}
	\item\label{I:AniShifDist:gradient-parallel} $\nabla \sdist^A(y) $ is parallel to $(\nabla \Phi^\ast)\left( y + \sdist^A(y)\omega\aniEn -x\right)$ and
	$\nabla\Phi ({\nabla\sdist^A(y)}) = \frac{y + \sdist^A(y)\omega\aniEn -x}{\sdist^A(y)} ,$

    \item\label{E:gradient-shifted_dist} $\left({1+\sfrac{|\omega|}{\Phi(\nu_H)}}\right)^{-1} \leq \Phi(\nabla\sdist^A(y)) \leq \left({1-\sfrac{|\omega|}{\Phi(\nu_H)}}\right)^{-1}$,
    
    \item\label{I:diff implies uniq proj} $y\in \bU_\omega^A$ and $x=\bp_\omega^A(y)$.
\end{enumerate}
\end{lemma}
\begin{proof}
We will drop $A$ from the notation in this proof. 
We first show \eqref{I:AniShifDist:homogeneity}. Note that $y + t(x-y) + (1-t)\sdist(y)\omega\aniEn =: W$ satisfies
\begin{equation*}
	y + \sdist(y)\omega\aniEn + t\left(x - (y + \sdist(y)\omega\aniEn)\right) = W.
\end{equation*}

Then, by positive one-homogeneity of $\Phi^\ast$ and the last displayed equality, we readily see that $\rmd(x,W) = \Phi^\ast( x - W) = |1-t|\sdist(y)$ which, by definition of $\sdist$, leads to $ \sdist( y + t(x-y)) \leq \rmd(x,W) \leq |1-t|\sdist(y)$ for any $t \in \R$. On the other hand, by a direct computation using convexity of $\Phi^\ast$ and the definition of $\aniEn$, we have
\begin{equation*}
\begin{aligned}
	\overline{\cW_{(1-t)\sdist(y)}(y + t(x-y) + (1-t)\sdist(y)\omega\aniEn)} \subset \overline{\cW_{\sdist(y)}(y + \sdist(y)\omega\aniEn)} \mbox{ and }\\
	\cW_r\left( y + t(x-y) + r\omega\aniEn\right) \subset \cW_{(1-t)\sdist(y)}(y + t(x-y) + (1-t)\sdist(y)\omega\aniEn),
\end{aligned}
\end{equation*}
for any $t \leq 1$ and $r\in  [0, (1-t)\sdist(y))$. For any given $r\in [0, (1-t)\sdist(y))$, this immediately implies that
\begin{equation*}
	\cW_r\left( y + t(x-y) + r\omega\aniEn\right)\cap A = \emptyset.
\end{equation*}

This guarantees that $\sdist(y + t(x-y)) \geq (1-t)\sdist(y)$ which finishes the proof of \eqref{I:AniShifDist:homogeneity}. 

To prove \eqref{I:linear distance}, we proceed as follows. First, applying \eqref{I:AniShifDist:homogeneity} with $\tau=1-s$ and $y = x+tu$, we get $\sdist(x+ tsu) = st$ for any $s\in (0,1]$. Assume by contradiction that there exists $a\in A\setminus\{x\}$ such that $\Phi^\ast(x+su+\sdist(x+su)\omega\aniEn - a) = \sdist(x+su) = s$. Notice that $a\neq x+\tau u + t\omega\aniEn $ for any $\tau\in (0,t)$ otherwise we would get the following contradiction $$t = \sdist(x+tu) \leq \Phi^\ast(x+tu + t\omega\aniEn - a) = \Phi^\ast((t-\tau)u) = t-\tau < t.$$

Therefore, we have that $su$ and $x-a+tu+t\omega\aniEn$ are linearly independent for any $s\in (0,t)$. By strict convexity of $\Phi^\ast$ and recalling $u\in \partial \cW_1 - \omega\aniEn$, we obtain 
$$ t = \sdist(x+tu) \leq \Phi^\ast(x+tu+t\omega\aniEn -a) < \Phi^\ast(x-a +su + s\omega\aniEn) + \Phi^\ast((t-s)(u+\omega\aniEn)) = s + t-s = t.$$
This gives a contradiction and ensures that $x$ is the unique point with $\sdist(x+su) = s$. This concludes the proof of \eqref{I:linear distance}.

Let us prove \eqref{I:AniShifDist:gradient-parallel}, thanks to \cite[Lemma 2.38 (d)]{de2020uniqueness}, we have that 
\begin{equation}\label{E:derivativeF*}
	(\nabla \Phi^\ast)\left(  \frac{- x + y + \sdist(y)\omega\aniEn}{\sdist(y)}\right) = (\nabla \rmd)\left( y + \sdist(y)\omega\aniEn\right).
\end{equation}

By differentiating $\sdist(y) = \rmd( y + \sdist(y)\omega\aniEn, A)$, we obtain that
\begin{equation*}
\begin{aligned}
	\nabla\sdist(y) &= \nabla\rmd( y + \sdist(y)\omega\aniEn ) \\
	&= (\nabla \rmd)(y + \sdist(y)\omega\aniEn) + \omega \left( \aniEn \cdot (\nabla \rmd)(y + \sdist(y)\omega\aniEn) \right)\nabla \sdist(y) \\
	\overset{\eqref{E:derivativeF*}}{=}& (\nabla \Phi^\ast)\left(  \frac{- x + y + \sdist(y)\omega\aniEn}{\sdist(y)}\right) + \omega\left (\aniEn \cdot (\nabla \rmd)(y + \sdist(y)\omega\aniEn)\right) \nabla \sdist(y).
\end{aligned}
\end{equation*}
This and the zero-homogeneity of $\nabla\Phi^*$  clearly imply the first affirmation in \eqref{I:AniShifDist:gradient-parallel}. The second affirmation in \eqref{I:AniShifDist:gradient-parallel} follows from the latter and $\nabla \Phi\circ\nabla \Phi^* (z) = z$ for any $z\in\partial\cW_1$, see \eqref{eqn: CH map}.

Next, we prove \eqref{E:gradient-shifted_dist}. By \eqref{I:AniShifDist:gradient-parallel}, we have
\begin{equation}\label{E:y-xTaylor}
	y - x = \sdist(y) \left( \nabla\Phi(\nabla\sdist(y)) - \omega\aniEn\right).
\end{equation}
Taylor expanding $\sdist(y+t(x-y))$ at $y$ gives
\begin{equation*}
(1-t)\sdist(y) \overset{\eqref{I:AniShifDist:homogeneity}}{=} \sdist(y+t(x-y)) = \sdist(y) + t \nabla\sdist(y)\cdot( x-y) + o(t),
\end{equation*}
which rearranging and taking the limit $t\to 0$ provides
\begin{equation}\label{E:fromTaylor}
	-\sdist(y) = \nabla\sdist(y)\cdot (x-y).
\end{equation}
By inserting \eqref{E:y-xTaylor} into \eqref{E:fromTaylor} and using one-homogeneity of $\Phi$ (namely that $\Phi(z) = \nabla\Phi(z)\cdot z$), we derive
\begin{equation*}
	1 = \nabla\sdist(y) \cdot \nabla\Phi(\nabla\sdist(y)) - \omega\nabla\sdist(y)\cdot \aniEn = \Phi(\nabla\sdist(y)) - \omega\nabla\sdist(y)\cdot \aniEn.
\end{equation*} 
Applying the Fenchel inequality  to the dot product in the right-hand side and the definition of ${\aniEn}$, we obtain \eqref{E:gradient-shifted_dist}.
Item \eqref{I:diff implies uniq proj} is a direct consequence of item \eqref{I:AniShifDist:gradient-parallel}. 
\end{proof}

\subsection{Proof of the Heinze-Karcher inequality}
We now proof Theorem~\ref{thm: HK}.

\begin{proof}[Proof of Theorem~\ref{thm: HK}]

Following the proof of \cite[Proof of Theorem 1.2]{jia2023alexandrov}, we consider the mapping 
\begin{equation}\label{eqn: zeta def}
\zeta(x,t) := x + t\big( -\nu_E^\Phi(x) - {\omega} \aniEn \big),
\end{equation}
which is classically defined and differentiable on the set 
\begin{equation}\label{eqn: z def}
Z:= \big\{(x,t)\in \reg(\rel E)\times \R:  0 < t \leq {\max_i\{\kappa_i^\Phi
    (x)\}}^{-1}\big\}\,.
\end{equation}
Direct computation shows that
\begin{equation}\label{E:Jacobian zeta}
	\mathrm{J}^Z \zeta (x, t)=\big(\Phi(\nu_E(x)) +{\omega} \nu_{E}(x)\cdot \aniEn \big) \prod_{i=1}^n\big(1-t \kappa_i^\Phi(x)\big).
\end{equation}
(Note that $\Phi(z) + \omega  z\cdot \aniEn > 0$ by the Fenchel inequality since $\omega\in (-\Phi(\nu_H) ,\Phi(\nu_H))$.)
We claim that
\begin{equation}\label{E:covering-property-zeta}
	|E \setminus \zeta(Z)| = 0.
\end{equation}

Once \eqref{E:covering-property-zeta} is shown, the proof of the Heintze-Karcher inequality \cite[Proof of Theorem 1.2]{jia2023alexandrov} goes through verbatim; we recall the proof here  before proving the claim. By \eqref{E:covering-property-zeta} and the area formula, we have
\begin{equation*}
	|E| \leq \left|\zeta (Z)\right| \leq \int_Z \cH^0(\zeta^{-1}(y)) \rmd  \cH^{n+1}(y)=\int_Z \mathrm{J}^Z \zeta \, \rmd  \cH^{n+1}.
\end{equation*}
Substituting \eqref{E:Jacobian zeta} into the right-hand side then shows
\begin{equation*}
\begin{aligned}
	|E| &\leq  \int_{\reg(\rel E)} \biggl( \int_0^{{\max_i\{\kappa_i^\Phi\}}^{-1}}\big(\Phi(\nu_E) +{\omega}  \nu_{E}\cdot \aniEn \big) \, \prod_{i=1}^n(1-t \kappa_i^\Phi)\,  \rmd  t \biggr)\rmd  \mathcal{H}^n.
\end{aligned}
\end{equation*}
By the AM-GM inequality and since $\max_i\{\kappa_i^\Phi(x)\} \geq H_E^\Phi(x) / n$, we readily obtain
\begin{equation*}
\begin{aligned}
	|E| &\leq \int_{\reg(\rel E)} \biggl( \int_0^{{\max_i\left\{\kappa_i^\Phi\right\}}^{-1}}\big(\Phi(\nu_E) +{\omega} \nu_{E}\cdot \aniEn  \big) \bigg(\frac{1}{n} \sum_{i=1}^n(1-t \kappa_i^\Phi)\bigg)^n \rmd  t\biggr)\rmd \cH^n \\
	& \leq \int_{\reg(\rel E)} \big(\Phi(\nu_E) +{\omega}\nu_{E}\cdot\aniEn  \big) \biggl( \int_0^{\frac{n}{H_E^\Phi}}\bigg(1-t \frac{H_E^\Phi}{n}\bigg)^n \rmd  t\biggr)\rmd \cH^n \\
	& =\frac{n}{n+1} \int_{\reg(\rel E)} \frac{\Phi(\nu_E) +{\omega} \nu_{E}\cdot \aniEn }{H_E^\Phi} \mathrm{~d} \mathcal{H}^n .
\end{aligned}
\end{equation*}
This establishes \eqref{eqn: HK}.

It remains to show \eqref{E:covering-property-zeta}.  We set $A:= \overline{H\setminus E}$ and  let $
	\Sigma_s:=\{y \in H : \sdist^A(y)=s\}$ for each $s>0$.
By the coarea formula and the fact that $|\nabla\sdist^A|\geq c_0 = c_0(\Phi,\omega) >0$ (see Lemma~\ref{L:AniShifDistBasicProps}\eqref{E:gradient-shifted_dist}), we have
\begin{equation*} 
\begin{aligned}
	c_0 |E\setminus\zeta(Z)| \leq \int_{E\setminus\zeta(Z)} |\nabla\sdist^A(y)| &= \int_0^{+\infty} \cH^n (\Sigma_s\setminus\zeta(Z))\, \rmd s.
\end{aligned}
\end{equation*}
So, to prove \eqref{E:covering-property-zeta}, it suffices to show $\cH^n (\Sigma_s\setminus\zeta(Z))=0$ for every $s>0$. Fix $s>0$, and for $x \in \Sigma_s$, let $c = x+ s {\omega}\aniEn$ and $G = \cW_s(c)$. With this notation, $G \cap H \subset  E$, and for any $a \in \partial G \cap \rel E$ and  ${u} := s^{-1}(x-a)$, we have  $(a,{u})\in N_\omega(A)$, and $\nu_{G}^\Phi(a)=s^{-1}(a-c).$

First, let $\Sigma_s^1$ be the set of those $x \in\Sigma_s$ for which there is $a \in \partial G \cap \Gamma$, and fix any such $x$ and $a$. We claim that $a \in T_E.$ If not, then since $E$ is a viscosity subsolution for Young's law and $G$ is a $C^1$ set touching $E$ from the interior at $a$, we have
\begin{equation}\label{eqn: angle cond HK}
	\nu_G^\Phi(x)\cdot \nu_H \leq -\omega.
\end{equation}
On the other hand,  $\nu^\Phi_{G}(a) = {s}^{-1}(a-c)$ as observed above.
Thus $\nu^\Phi_{G}(a) \cdot \nu_H = -s^{-1}c\cdot \nu_H =-{s}^{-1} x\cdot \nu_H -  \frac{\omega}{\Phi(\nu_H) }\nu_H^\Phi\cdot \nu_H > -\omega$, where in the strict inequality we use $x \in H$. This contradicts \eqref{eqn: angle cond HK}. 
Thus $a \in T_E$, and in turn $x = a- (1-{\omega} \, s) \aniEn.$ In other words, $\Sigma_s^1$ is a translation of a subset of $T_E$, so
the assumption \eqref{A:RegAnis:tang touching pts} that  $\cH^n(T_E)=0$ guarantees that $\cH^n (\Sigma_s^1) = 0$.

Next, suppose $x \in \Sigma_s^2 :=\Sigma_s \setminus \Sigma_s^1$ and fix $ a\in\rel E\cap H\cap \partial G$. If $a\in\reg(\rel E)$ then since $G$ touches $E$ from the interior at $a$, we have  $s^{-1} = \kappa^{\Phi,i}_{\partial G}(a) \geq \kappa_{\partial E}^{\Phi, i}(a)$ for any $i$. Hence $(a,s)\in Z$ and therefore, $p = \zeta(a,s)\in \zeta(Z)$. Consequently, if $x \in \Sigma_s^2 \setminus \zeta(Z)$, then $a \in \sing(\rel E) \cap H$. In particular,  since $x=a+su$ with $u$ as above, this means $\Sigma_s^2 \setminus \zeta(Z)$ is contained in $(N_\omega(A)|(\sing(\rel(E)) \cap H)$. 
From assumption \eqref{A: HK normal bundles meas zero},  it follows that  $\cH^n (\Sigma_s^2\setminus\zeta(Z))=0$. Together with $\cH^n (\Sigma_s^1) = 0$, this completes the proof of \eqref{E:covering-property-zeta}. 
\end{proof}

\section{Proof of the Alexandrov Theorems}\label{S:proof of Alexandrov}
The main goal of this section is to prove Theorems~\ref{T:Alexandrov:Isotropic}, \ref{T:new theorem}, and \ref{T:Alexandrov:Anisotropic}. First, we prove some preparatory lemmas.

\begin{proposition}\label{P: Steiner}
Let $A \subseteq \overline{H}$ be closed and $r>0$. Suppose that, for every $\cH^n$-measurable bounded function $f: \overline{H} \times (\partial \cW_1 {-\omega\aniEn} )\to \R$ with compact support, there are numbers $c_1(f), \ldots, c_{n+1}(f) \in \R$ such that
\begin{equation}\label{eqn: hp steiner}
\int_{\left( H\setminus A\right)\cap \left\{x: \sdist^A(x) \leq \rho\right\}} f\left(\bp_\omega^A(x), \nabla\Phi (\nabla \sdist^A(x)){-\omega\aniEn}\right) \rmd x=\sum_{j=1}^{n+1} c_j(f) \rho^j \quad \text { for any } \rho\in (0,r) .
\end{equation}
Then 
\begin{equation}
    \label{eqn: reach LB}
    \reach(A) := \sup\left\{ s \geq 0 : \{y\in \overline{H}: \sdist^A(y) < s\}\subset \bU_\omega^A = \mathrm{dmn}(\bp_\omega^A)\right\} \geq r .
    \end{equation}
\end{proposition}
\begin{proof}
The proof goes along the same lines of \cite[Theorem 5.9]{de2020uniqueness}, we present it here for completeness.
Since $A$ is fixed, in the proof we write $\sdist, p$ and $U$ in place of $\sdist^A, \bp_\omega^A(x)$, and $\bU_\omega^A$ respectively.
Let $r(a,u) := \sup\{s: \sdist(a+su)=s\}$ for any $(a,u)\in N_\omega(A)$.

First, observe that it suffices to show
\begin{equation}
    \label{eqn: claim 1}
    \big|\big\{x\in \overline{H} : 0<\sdist(x) \leq r \mbox{ and } r\big(p(x), \nabla\Phi(\nabla \sdist(x)) - \omega\aniEn \big)<r\big\}\big|=0
\end{equation}
Indeed, suppose \eqref{eqn: claim 1} holds and 
fix $x\in \overline{H}\setminus A$ with $0 < \sdist(x) < r$. By \eqref{eqn: claim 1} and the fact that $\sdist$ is differentiable a.e., we can find a sequence $x_i \in U$ converging to $x$ with 
\begin{equation}\label{eqn: smart choice of x_i}
    0 < \sdist(x_i) \leq r\qquad\mbox{ and }\qquad r\big(p_i, v_i - \omega\aniEn \big) \geq r, 
\end{equation}
where $p_i:=p(x_i)$ and $v_i:=\nabla\Phi(\nabla\sdist(x_i)) = {\sdist(x_i)}^{-1}(x_i + \sdist(x_i) \omega \aniEn - p_i)\in\partial \cW_1$ (c.f. Lemma~\ref{L:AniShifDistBasicProps}\eqref{I:AniShifDist:gradient-parallel}).
As $\{p_i\}_i\subset A$ and $\{v_i\}_i$ are bounded, we pass to a subsequence to obtain limits $p_i \to p\in A$ and $v_i \to v\in \partial \cW_1$. By continuity of $\sdist,$  we have $x = p + \sdist^A(x)(v-\omega\aniEn)$ and 
$r(p, v-\omega\aniEn ) \geq r$.
Then, since $\sdist^A(x)<r$,  Lemma~\ref{L:AniShifDistBasicProps}\eqref{I:linear distance} guarantees that $x \in U$, completing the proof of \eqref{eqn: reach LB}.

Toward proving \eqref{eqn: claim 1}, we first let 
$$S=\left\{(a, u, t) \, :\, (a, u) \in N_\omega(A),\  t \in (0,   r(a, u))\right\}. $$
and  show that
\begin{equation}\label{eqn: old claim 1 Steiner}
   |\R^{n+1} \setminus(A \cup \phi(S))|=0\,.
\end{equation}
where $\phi: N_\omega(A) \times(0, +\infty) \rightarrow \R^{n+1}$ be  defined by $\phi(a, u, t)=a+t u$. Once again, since $\sdist$ is differentiable a.e., it suffices to show the claim with $U$ in place of $\R^{n+1}.$
Note that if  $x\in U \setminus A$, then $r\big(p(x), \frac{x - p(x)}{\sdist(x)}\big) \geq \sdist^A(x)$, while if $x \in U \setminus\phi(S)$, then   $r\big(p(x), \frac{x - p(x)}{\sdist(x)}\big) \leq \sdist(x)$, and thus
\begin{equation*}
   U \setminus (A \cup \phi(S)) \subseteq \phi\left(\left\{(a, u, t)\ : \ (a, u) \in N_\omega(A),\  t=r(a, u)>0\right\}\right) .
\end{equation*}
Since $\phi$ is locally Lipschitz, to prove \eqref{eqn: old claim 1 Steiner}, it suffices to prove that, for any $K \subset N_\omega(A)$ bounded and $M\in \mathbb{N}$, 
\begin{equation}\label{eqn: reduction Steiner2}
    \mathcal{H}^{n+1}(\left\{(a, u, t) \ : \ (a, u) \in K, \ t=r(a, u) \in (0, M) \right\})=0 .
\end{equation}
To this aim, fix $K\subseteq N_\omega(A)$ bounded and $n$-rectifiable (so that $\mathcal{H}^n(K)<\infty$), $M\in\mathbb{N}$, and for every $q \in \R$,  we define the Borel set 
$$ V_q=\left\{(a, u, t+q)\ : \ (a, u) \in K,  \  t=r(a, u)\in (0, M ) \right\}.$$ 
Since $N_\omega(A)$ is $n$-rectifiable (see Lemma~\ref{L: Lusin cond - shifted vs classical}), to prove \eqref{eqn: reduction Steiner2}, it is in turn enough to prove $\mathcal{H}^{n+1}(V_0)=0.$
Observe that $V_q \cap V_p=\varnothing$ whenever $p\neq q$, therefore 
\begin{equation*}
\begin{aligned}
 +\infty>   (M+1)\cH^n(K) = \cH^{n+1}\left(K\times (0,M+1)\right) &\geq \cH^{n+1}\Big(\cup_{q\in (0,1)\cap\mathbb{Q}} V_q\Big) \\
    &= \sum_{q\in (0,1)\cap\mathbb{Q}} \cH^{n+1}(V_q) = \sum_{q\in (0,1)\cap\mathbb{Q}} \cH^{n+1}(V_0).
\end{aligned}
\end{equation*} 
It follows that $\mathcal{H}^{n+1}(V_0) = 0$, and thus \eqref{eqn: old claim 1 Steiner} holds.

Now, since  $N_\omega(A)$ is $n$-rectifiable (Lemma~\ref{L: Lusin cond - shifted vs classical}), we can choose a partition $N_\omega(A)=\bigcup_{i=1}^{\infty} N_i$ such that each $N_i$ is a $n$-rectifiable set and $\cH^n(N_i)<\infty$ (see \cite[2.1.6]{federer2014geometric}). For $i \in \mathbb{N}$, we define
$$ S_i := S \cap\left(N_i \times(0, +\infty)\right) \mbox{ and } J(a,u,t) := \sum_i J^{N_i} \phi(a,u,t),  $$
whenever the tangential Jacobian exists. For a given $g: \R^{n+1} \rightarrow \R$ non-negative Borel function with compact support contained in $\overline{H}$, noticing that $\phi |_S$ is injective by Lemma~\ref{L:AniShifDistBasicProps}\eqref{I:linear distance}, we can apply the coarea formula (\cite[3.2.22]{federer2014geometric}) to $g\circ\phi$ to find
\begin{equation}\label{E: coarea in Steiner}
\begin{aligned}
    \int_{H \setminus A} g (x)\rmd x \overset{\eqref{eqn: old claim 1 Steiner}}{=} & \int_{\phi(S)} g(x) \, \rmd x=\sum_{i=1}^{\infty} \int_{\phi\left(S_i\right)} g(x)\,  \rmd x \\
    &= \sum_{i=1}^{\infty}\int_{S_i} g(a+t u) J^{N_i}\phi(a, u, t) \, \rmd \cH^{n+1}(a, u, t)\\
    &= \sum_{i=1}^{\infty} \int_0^{\infty} \int_{N_i\cap \left\{(w, v): r(w, v)>t\right\}} g(a+t u) J^{N_i}\phi(a, u, t)\,  \rmd \cH^n(a, u)\,  \rmd t \\
    &= \int_0^{\infty}\int_{\left\{(w, v): r(w, v)>t\right\}}  g(a+t u)J(a, u, t) \, \rmd \cH^n(a, u) \, \rmd t.
\end{aligned}    
\end{equation}

Let $B \subseteq A$ be compact, $0<\tau < \rho < r$, and define $N_{\tau, B}=N_\omega(A) \cap\left\{(a, u): r(a, u) \leq \tau, a \in B\right\}$. Consider the function $g=(\mathbf{1}_{N_{\tau, B}} \circ (p, \nabla\Phi\circ\nabla\sdist-\omega\aniEn))\mathbf{1}_{\left\{w\in \overline{H}: \sdist(w) \leq \rho\right\}}$. The assumption \eqref{eqn: hp steiner} applied to $f =\mathbf{1}_{N_{\tau, B}} $ guarantee there are constants $c_1,\dots, c_{n+1}$ such that
\begin{equation}\label{eqn: assump Steiner application}
\begin{aligned}
    \sum_{j=1}^{n+1}c_j\rho^j = \int_{\left(H\setminus A\right)\cap \left\{x: \sdist(x) \leq \rho \right\}} \mathbf{1}_{N_{\tau, B}} (p(x), \nabla\Phi(\nabla\sdist(x))-\omega\aniEn)\,  \rmd x = \int_{H\setminus A}g(x) \, \rmd x.
\end{aligned}    
\end{equation}
Applying \eqref{E: coarea in Steiner} to the right-hand side and recalling $\sdist(a+tu) = t < r(a, u) \leq \tau<\rho$, this implies
\begin{equation*}
\begin{aligned}
    \sum_{j=1}^{n+1}c_j\rho^j &=  \int_0^{\infty}\int_{\left\{(w, v): \tau \geq r(w, v)>t, w\in B\right\}} J(a, u, t) \mathbf{1}_{\left\{w: \sdist(w) \leq \rho \right\}}(a+tu) \,  \rmd \cH^n(a, u)\, \rmd t\\
    &= \int_0^{\infty}\int_{\left\{(w, v): \tau \geq r(w, v)>t, w\in B\right\}} J(a, u, t) \, \rmd \cH^n(a, u) \, \rmd t,
\end{aligned}
\end{equation*}
As the right-hand side is independent of $\rho \in (\tau, r)$, and thus $c_1 = \dots, c_{n+1} = 0$.
Therefore, by \eqref{eqn: assump Steiner application}, we have
$$\big|\big\{x \in \overline{H} \, :\,  0<\sdist^A(x) \leq r\mbox{ and } (p(x), \nabla \Phi(\nabla\sdist(x))-\omega\aniEn) \in N_{\tau, B}\big\}\big|=0 . $$
Since this holds for every $0<\tau<r$ and for every compact set $B \subseteq A$, \eqref{eqn: claim 1} follows.
\end{proof}

\begin{lemma}\label{L: Hausdorff BU}
Given $A\subseteq \overline{H}$ closed and $x\in\mathbf{q}(N_\omega(A))$ where $\mathbf{q}:(a,u)\in N_\omega(A)\to a\in A$, there is no blow-up of $A$ at $x$ (in the Hausdorff topology) that coincides with $\R^{n+1}$.
\end{lemma}
\begin{proof}
We observe that, by definition, there exists $t>0$ and $u\in\partial \cW_1 - \omega\aniEn$ such that $\sdist^A(x+tu) = t$ and, thus, by Lemma~\ref{L:AniShifDistBasicProps}\eqref{I:linear distance}, we have $\sdist^A(x+su) = s$ for any $s\in (0,t]$. This implies that
$$ A\cap \cW_{\tau}\left( x+ \tau v\right) = \varnothing, \mbox{ for any }\tau \in (0, t)\mbox{ and where } v = u + \omega\aniEn \in \partial\cW_1. $$

Otherwise, if $a \in A\cap \cW_{\tau_0}\left( x+ \tau_0 v\right)$ for some $\tau_0\in (0, t)$, it would imply $\tau_0 = \sdist^A(x+\tau_0 u) \leq \Phi^*(x+\tau_0 v - a) < \tau_0$ and we would have achieved a contradiction. 

Next, take any blowup sequence $\{r_k\}_k$ and assume without loss of generality $r_k\in (0,t)$ for any $k$. By the last displayed equation with $\tau = r_k$, we obtain
$$ A_{x,r_k} \cap \cW_1(v) = \varnothing\mbox{ for every }k\in\mathbb{N} .$$

Since $\cW_1(v)$ is open, this implies that the limit in the Hausdorff topology of $A_{x,r_k}$ cannot be the whole $\R^{n+1}$, as desired.
\end{proof}

We show how Theorem~\ref{T:Alexandrov:Anisotropic} follows by combining Theorem~\ref{thm: stationary implies viscosity} and Theorem~\ref{thm: HK}.
We will need the following definition.
\begin{definition}
    We say that $Z \subset H$ is a $(n, h)$-set in $H$ with respect to $\Phi$ if $Z$ is relatively closed in $H$ and for any open set $G \subset H$ such that $\partial G \cap H$ is smooth and $Z \subset \overline{G}$,
\[
     H^\Phi_G(p) \leq  h, \quad \mbox{ for any } p \in Z \cap \partial G \cap H.
\]
\end{definition}

\begin{proof}[Proof of Theorem~\ref{T:Alexandrov:Anisotropic}]
Let $E$ be as in Theorem~\ref{T:Alexandrov:Anisotropic}, so $E$ satisfies \eqref{eqn: aniso isotropic stationary} for some $\lambda>0$.

{\it Step 1:} 
First, we show that $E$ satisfies the assumptions of Theorem~\ref{thm: HK}. 
The density lower bound \eqref{eqn: lower density assumption intro 1} ensures that we can apply Theorem~\ref{thm: stationary implies viscosity}, and thus $E$ is a viscosity subsolution to Young's law. 
Since $E$ is satisfies \eqref{eqn: aniso isotropic stationary}, the anisotropic mean curvature $H_E^\Phi$ is equal to $\lambda$, hence positive, on $\reg(\rel E)$. This together with the assumption \eqref{A:RegAnis:tang touching pts} guarantee that \eqref{A:HK} holds.

We now show that $E$ satisfies assumption \eqref{A: HK normal bundles meas zero}. 
Let $A:= \overline{H \setminus E}\subseteq \overline{H}.$
Notice that $\mathbf{q}(N_\omega(A))\subset \partial A$ by the definition of $N_\omega(A)$. 
Next, thanks to \eqref{A:RegAnis:Intro} and Lemma~\ref{L: Hausdorff BU}, we can invoke \cite[Theorem 4.1]{santilli2025quadratic}, which relies on  Allard's $\epsilon$-regularity theorem  \cite{allard1986integrality,santilli2025quadratic}, to deduce that $\cH^n$-a.e. point of $\mathbf{q} (N_\omega(A)) {\cap H}$ is a $C^{1,\alpha}$-regular point of $\partial A {\cap H} = \partial E \cap H$ (cf. \cite[Remark 4.2]{santilli2025quadratic}). Hence, by \cite[Remark 6.3]{de2020uniqueness}, $\cH^n$-a.e. point of $\mathbf{q} (N_\omega(A)) \cap H$ is a $C^{2,\alpha^2}$-regular point of $\partial E \cap H$.

By \cite[Lemma 4.11]{de2020uniqueness},  $\partial A\cap H = \partial E\cap H$ is an $(n,\lambda)$-set in $H$ with respect to $\bPhi$.  
Thanks to Lemma~\ref{L: Lusin cond - shifted vs classical}, we deduce from \cite[Theorem 4.10 and Lemma 5.4]{de2020uniqueness} that $N_\omega (A)$ satisfies the Lusin condition in $H$.
So, by observing that 
$$N_\omega(A)\mid (\partial A \cap H\setminus \reg(\rel E))=N_\omega(A)\mid ((\mathbf{q} (N_\omega(A)) \cap H) \setminus \reg(\rel E)),$$
we conclude from the Lusin condition in $H$ that
\begin{equation}\label{eqn: Lusin in rigidity}
    \cH^n(N_\omega(A)\mid (\partial A \cap H\setminus \reg(\rel E)))=0.
\end{equation}
This shows that \eqref{A: HK normal bundles meas zero} is satisfied, and therefore Theorem~\ref{thm: HK} applies to $E$.

{\it Step 2:}
Next, choose a sequence $R_j \to \infty$ so that 
\begin{equation}\label{eqn: perimeter decay}
  R_j \, (P(E;B_{R_j +1}\setminus B_{R_j}) + |E \cap (B_{R_j +1}\setminus B_{R_j})|) \to 0;
\end{equation}
it is not difficult to show such a sequence exists by the finiteness of perimeter and volume.
Consider the vector fields 
 $X_j(x):= \varphi_j(x)x$ where $\varphi_j\in C_c^\infty (\R^{n+1}, [0,1])$ is a cutoff function with $\varphi_j \equiv 1$ in $B_{R_j}$, $\varphi_j \equiv 0$ in $\R^{n+1}\setminus B_{R_j+1}$, and $\|\nabla \varphi_j\|_{L^{\infty}} \leq 2$.
 Testing  the first variation \eqref{eqn: aniso isotropic stationary} with $X_j$, applying the divergence theorem, and passing to the limit using \eqref{eqn: perimeter decay} shows that  for any $x\in\reg(\rel E)$,     
\begin{equation*}
    H_E^\Phi(x) \equiv \lambda = \frac{n}{(n+1)|E|}\int_{\rel E} \Phi(\nu_E) + \omega \nu_E\cdot \bar{\nu}_H^\Phi\rmd \cH^n.
\end{equation*}

Substituting this into the right hand side of the Heintze-Karcher inequality \eqref{eqn: HK}, which we may apply by Step 1, we see that equality holds in each step of the proof of  Theorem~\ref{thm: HK}. In particular, {$|\zeta(Z) \setminus E| =0 $ where $\zeta$ and $Z$ are as in \eqref{eqn: zeta def} and \eqref{eqn: z def}, }  
  $\cH^0(\zeta^{-1}(y)) = 1$ for almost every $y\in E$, where $\zeta$ is the function defined in \eqref{eqn: zeta def}, and   $\kappa_i^\Phi(x)=\lambda / n$ for any $i$ and $x\in\reg(\rel E).$ 
  Consequently, letting $p=\bp_\omega^A$ and $\sdist = \sdist^A$ where $A = \overline{H\setminus E}$, the function $p(y)$ is defined, with $(p(y), d_\w(y)) =\zeta^{-1}(y)$, and lies in $\reg(\rel E)$ for a.e. $y \in E$. At any such point, $\nabla\Phi(\nabla \sdist(y)) = \frac{c(y) -p(y)}{\sdist(y)} = - \nu_E^\Phi(p(x))$  by Lemma~\ref{L:AniShifDistBasicProps}\eqref{I:AniShifDist:gradient-parallel}, where $c(y) = y + \sdist(y) \omega \aniEn$.

For $\tau\in (0,n/\lambda)$, let $Z_\tau := \reg(\rel E)\times (0, \tau]$.
 Using \eqref{eqn: Lusin in rigidity} and arguing as in the proof of Heintze-Karcher inequality (Theorem~\ref{thm: HK}), we obtain 
 \begin{equation}
     \label{eqn: Ztau cover}
|E\cap \{x: \sdist(x) \leq \tau\}\setminus \zeta(Z_\tau) | = 0.
 \end{equation}
{Since $|\zeta(Z) \setminus E| = 0$ and $Z_\tau \subset Z$, we have $|\zeta(Z_\tau) \setminus E| = 0$ as well}.
Therefore, for an $\cH^n$-measurable and bounded function $f : \overline{H} \times \partial \cW_1 - \omega \aniEn \to \R$ with compact support,  \eqref{eqn: Ztau cover} and the area formula guarantee that
\begin{equation*}
\begin{aligned}
\int_{E \cap \{y: \sdist(y) \leq \tau\}} &f(p(x), \nabla\Phi(\nabla \sdist(x)-\omega \aniEn))\, \rmd x
    = \int_{E \cap \zeta(Z_\tau)}f(p(x), \nabla\Phi(\nabla \sdist(x)) -\omega \aniEn)\, \rmd x\\
 &{=} \int_{Z_\tau}f(z,-\nu_E^\Phi(z) -\omega \aniEn)J^{Z_\tau}\zeta(z,t)\, \rmd\cH^{n+1}(z,t) \\
    &= \int_0^\tau \left( 1 - t\lambda\right)^n\rmd t\int_{\reg(\rel E)\cap H}f(z,-\nu_E^\Phi(z)-\omega \aniEn)\left(\Phi(\nu_E(z)) + \omega \nu_E^\Phi(z)\cdot \aniEn\right)\rmd\cH^n(z),
\end{aligned}
\end{equation*}
where in the last equality we used the formula for the Jacobian of $\zeta$ (see \eqref{E:Jacobian zeta}), $\kappa^\Phi_i(z)  \equiv n/\lambda$, and $\tau < n/\lambda$. 
Evaluating the integral in $t$ shows that the assumptions of Proposition~\ref{P: Steiner} hold with $r =n/\lambda$, and consequently $\reach(A) \geq n/\lambda > 0$.

{\it Step 3:}
Since $\reach(A) \geq n/\lambda$, the function $\sdist$ is differentiable on $ S:= \{y\in H:\sdist(y) < n/\lambda \}$, and   $\nabla\sdist \neq 0$ on $S$  by Lemma~\ref{L:AniShifDistBasicProps}\eqref{E:gradient-shifted_dist}.
So, for any  $r\in (0,n/\lambda)$, the set $T_r = (\sdist)^{-1}(\{r\})\cap H$ is a $C^1$-hypersurface in $H$ with $\reach(T_r)>0$. Arguing as in \cite[Lemma 5.7]{de2020uniqueness}, using that $\Phi\in C^{1,1}$ is uniformly elliptic and $\reach(T_r) >0$, we find that  $T_r$ is of class $C^{1,1}$ and without boundary in $H$.

We claim that $T_r$ has constant anisotropic mean curvature. For any $z\in T_r$, by Lemma~\ref{L:AniShifDistBasicProps}, we have $\nu_{T_r}^\Phi(z) =  -\nabla \Phi (\nabla \sdist(z)) = r^{-1}( p(z)-(z+r\omega\aniEn))$ and thus $\nu_{T_r}^\Phi(z) = \nu^\Phi_E(p(z))$. Therefore, $D(\nu_{T_r}^\Phi)_{z} = D(\nu^\Phi_E)_{p(z)}Dp_z$.

Next, let $\zeta^\prime(z) := \zeta (z,r)$, again with $\zeta$ as in \eqref{eqn: zeta def}. As observed above, $p(\zeta^\prime(z)) = z$ for any $z\in \reg(\rel E)$. Therefore, $p = (\zeta^\prime)^{-1}$ and $Dp_z = (D\zeta^\prime_z)^{-1}$. From the definition of $\zeta$, this implies
\begin{equation}\label{eqn: 2nd ff}
D(\nu_{T_r}^\Phi)_{z} = D(\nu^\Phi_E)_{p(z)}Dp_z = D(\nu^\Phi_E)_{p(z)} \left(\mathrm{Id}_{(n+1)\times (n+1)} - rD(\nu^\Phi_E)_{p(z)}\right)^{-1} = \frac{\lambda}{1 - r\frac{\lambda}{n}} \left(\begin{array}{cc}
  \text{Id}_{n \times n} & 0   \\
   0  & 0
\end{array} \right) ,
\end{equation}
where the final equality holds in  a suitable basis (see the discussion after \eqref{E:MC 0-level set}).
Thus anisotropic principal curvatures of $T_r$ are given by $\lambda_r = \sfrac{\lambda}{1 - r\frac{\lambda}{n}}$.

{\it Step 4:}
Finally, we conclude as follows. 
Let $\{T_r^k\}_k$ be the connected components of $T_r$. By \eqref{eqn: 2nd ff},  we have $D_{T_r} (\nu^\Phi_{T_r}(z) - \lambda_rz)=0$ where $D_{T_r}$ denotes the derivative restricted to the tangent space of $T_r$, so the vector field $c_k =\frac{1}{\lambda_r}( \nu^\Phi_{T_r^k}(z) - \lambda_r z)$ is constant on every connected component $T_r^k$. Since $\Phi^*\circ\nabla\Phi = 1$, 
\[ \Phi^*(z - c_k) =  \frac{1}{\lambda_r}\Phi^*(\nu^\Phi_{T_r^k}(z)) = \frac{1}{\lambda_r},\mbox{ for any }z\in T_r^k.\]

Since each $T_r^k$ has no boundary in $H$, we can use the computation above to ensure that $T_r^k$ is either an entire Wulff shape or an $\omega'$-Winterbottom shape (possibly with $\omega' \neq \omega$) in $H$, hence $T_r$ is a union thereof. 

Since $\reach(A)\geq n/\lambda >0$ and $r<n/\lambda$, we have $\bp_\omega^{A} ( T_r ) = \rel E$. Sending $r\to 0$, and using that  $E$ has  finite perimeter, we see that $E$ is a \emph{finite} union of Wulff shapes entirely contained in $H$ and $\omega'$-Winterbottom shapes (possibly with $\omega' \neq \omega$) all with the same radius $1/\lambda$. Finally, integrating by parts in \eqref{eqn: aniso isotropic stationary} shows that $\omega'=\omega.$  This concludes the proof of the theorem.
\end{proof}

Next, to prove Theorem~\ref{T:Alexandrov:Isotropic}, we need the following lemma.

\begin{lemma}\label{L:zerotouch}
    Let $\omega \in (-1/2,0]$ and let $E$ be a set of finite perimeter satisfying \eqref{eqn: isotropic stationary}. Then $\mathcal{H}^n(\Gamma)=0$. 
\end{lemma}
\begin{proof}
    Assume for the sake of contradiction that $\cH^n(\Gamma)>0$, thus we can pick $d\in \Gamma^{(1)}\cap \Gamma$, i.e., $d\in \Gamma$ is a density $1$ point of $\Gamma$ with respect to $\cH^n$. Thanks to the fact that $d\in\rel E$, we can find a sequence $\{d_k\}_{k\in\mathbb{N}}\subset \partial^* E\cap H$ converging to $d$. We claim that
\begin{enumerate}[(A)]
	\item $\Theta^n(\partial^* E\cap H, d) = 0$,
	\item $\Theta^n(\partial^* E\cap H, d_k) = 1$ for any $k$,
    \item $\Theta^n(\partial^* E\cap\partial H, d) =1$.
\end{enumerate}

Since $d_k\in\partial^* E$ for all $k$, (B) follows from finite perimeter set theory (see \cite[Corollary 15.8]{maggi2012sets}). 
Analogously, $\Theta^n(\partial^* E, d)=1$ since $d \in \partial^*E$ by assumption, and using this, it is easy to show that (C) implies (A). Moreover, since $d\in \Gamma^{(1)}$, we have  $1=\Theta^n(\Gamma,d) \leq \Theta^n(\partial E\cap\partial H, d) \leq 1$, hence (C) follows. 

In \cite[Theorem 3.2 and Corollary 5.1]{kagaya2017fixed} it is shown, as a consequence of a suitable monotonicity formula, that for $\omega\leq 0$,
at every point $x \in \rel E,$ the quantity
\[
\tilde{\Theta}(x) := \lim_{r\to 0} \frac{P(E; H \cap B_r(x) + P(E; H\cap B_r(\tilde{x})) - 2 \omega \mathcal{H}^n(B_r(x) \cap \partial^*E \cap \partial H) }{\omega_n r^n}
\]
is defined, and is upper-semicontinuous as a function of $x$. Here, $\tilde{x}$ is the reflection of $x$ across $\partial H$. Observe that for $d_k$ above, $\tilde{\Theta}(d_k) = 1$, and thus by uppersemicontinuity, $\tilde{\Theta}(d) \geq 1.$ On the other hand, by (A) and (C), we have 

\begin{equation}\label{e:Te}
	\tilde{\Theta}(d) \overset{\textup{(C)}}{=}  \lim_{r\to 0}\frac{2\per{E}{H\cap B_r(d)}}{\omega_n r^n} - 2\omega \overset{\textup{(A)}}{=} -2\omega.
\end{equation}
Since $-\omega<1/2$, this is a contradiction.
\end{proof}

\begin{proof}[Proof of Theorem~\ref{T:Alexandrov:Isotropic}]
Since we proved above Theorem~\ref{T:Alexandrov:Anisotropic}, in order to prove Theorem~\ref{T:Alexandrov:Isotropic}, we only need to prove that all assumptions in Theorem~\ref{T:Alexandrov:Anisotropic} are satisfied for the area functional and for $\omega \in (-1/2, 0]$.

First,  the fact that the topological and reduced boundary are $\cH^n$-equivalent in $H$, i.e., \eqref{A:RegAnis:Intro}, is a well-known fact that follows from density estimates, which in turn are a consequence of the monotonicity formula; see \cite[Corollary 17.18]{maggi2012sets}.
Similarly,  the uniform lower bound on the mass ratio in a neighborhood of $\Gamma$, i.e., \eqref{eqn: lower density assumption intro 1}, holds with $\e=1$ again thanks to the monotonicity formula.
Finally, the fact that $\cH^n(T_E)=0$, i.e., \eqref{A:RegAnis:tang touching pts}, is an obvious consequence of Lemma Lemma~\ref{L:zerotouch} and the fact that $T_E\subset \Gamma$. 
\end{proof}

\begin{proof}[Proof of Theorem~\ref{T:new theorem}]
    The proof follows verbatim the proof of Theorem~\ref{T:Alexandrov:Isotropic}, with the only difference being that we cannot prove that $\mathcal{H}^n(\Gamma)=0$ with the same contradiction argument in \eqref{e:Te}.
    However this is now an assumption of the theorem for $\omega \in (-1,0]$.
\end{proof}

            \bibliographystyle{siam}
            \bibliography{biblio}

\footnotesize{
Antonio De Rosa,\\
\emph{Department of Decision Sciences and BIDSA, Bocconi University, Milan, Italy,}\\
\emph{Email address:} antonio.derosa@unibocconi.it,}

\footnotesize{
Robin Neumayer,\\
\emph{Department of Mathematical Sciences, Carnegie Mellon University, Pittsburgh, PA 15213, USA,}\\
\emph{Email address:} neumayer@cmu.edu,}

\footnotesize{
Reinaldo Resende,\\
\emph{Department of Mathematical Sciences, Carnegie Mellon University, Pittsburgh, PA 15213, USA,}\\
\emph{Email address:} rresende@andrew.cmu.edu.
}
\end{document}